\documentclass[a4paper,12pt,reqno]{amsart}

\pdfoutput=1

\usepackage[margin=2.3cm, headsep=0.75cm, footskip=0.75cm]{geometry}
\usepackage{amsmath, amsthm, amssymb, mathtools}
\usepackage[hang, flushmargin]{footmisc}
\usepackage{enumitem}
\usepackage{xcolor, tikz, tikzsymbols, tikz-cd}
\usepackage[colorlinks=true, linkcolor=blue, linkbordercolor=blue, citecolor=red, citebordercolor=red, urlcolor=indigo, linktocpage=true]{hyperref}
\usepackage[capitalise, nameinlink, noabbrev, nosort]{cleveref}

\definecolor{indigo}{HTML}{492DA5}

\allowdisplaybreaks

\setenumerate{listparindent=\parindent}

\providecommand{\noopsort}[1]{}

\newenvironment{acknowledgements}{\begin{quotation}\footnotesize{\textsc{Acknowledgements.}}}{\end{quotation}}

\makeatletter\g@addto@macro\bfseries{\boldmath}\makeatother

\makeatletter
\let\origsection\section
\renewcommand\section{\@ifstar{\starsection}{\nostarsection}}
\newcommand\sectionspace{\vspace{0.5ex}}
\newcommand\nostarsection[1]{\sectionspace\origsection{#1}\sectionspace}
\newcommand\starsection[1]{\sectionspace\origsection*{#1}\sectionspace}
\makeatother

\setlist[enumerate]{font=\normalfont}

\crefname{enumi}{}{}
\crefname{enumii}{}{}
\crefname{condition}{condition}{conditions}

\newcommand\numberthis{\addtocounter{equation}{1}\tag{\theequation}}

\numberwithin{equation}{section}
\crefname{equation}{equation}{equations}

\newtheorem{theorem}{Theorem}[section]

\newtheorem{thm}[theorem]{Theorem}
\crefname{thm}{Theorem}{Theorems}

\newtheorem{lemma}[theorem]{Lemma}
\crefname{lemma}{Lemma}{Lemmas}

\newtheorem{prop}[theorem]{Proposition}
\crefname{prop}{Proposition}{Propositions}

\newtheorem{cor}[theorem]{Corollary}
\crefname{cor}{Corollary}{Corollaries}

\theoremstyle{definition}

\newtheorem{definition}[theorem]{Definition}
\crefname{definition}{Definition}{Definitions}

\theoremstyle{remark}

\newtheorem{remark}[theorem]{Remark}
\crefname{remark}{Remark}{Remarks}

\newtheorem{remarks}[theorem]{Remarks}
\crefname{remarks}{Remarks}{Remarks}

\newtheorem{example}[theorem]{Example}
\crefname{example}{Example}{Examples}

\newcommand{\hl}[1]{\textcolor{magenta}{\emph{#1}}}
\newcommand{\C}{\mathbb{C}}
\newcommand{\F}{\mathbb{F}}
\newcommand{\T}{\mathbb{T}}
\newcommand{\UU}{\mathcal{U}}
\newcommand{\Go}{{G^{(0)}}}
\newcommand{\Gc}{{G^{(2)}}}
\newcommand{\Sigmao}{{\Sigma^{(0)}}}
\newcommand{\Sigmac}{{\Sigma^{(2)}}}
\newcommand{\KP}{\operatorname{KP}}
\newcommand{\Rspan}{\operatorname{span}_R}
\newcommand{\supp}{\operatorname{supp}}
\newcommand{\id}{\operatorname{id}}
\newcommand{\Iso}{\operatorname{Iso}}
\newcommand{\lav}{\ensuremath{\lvert}}
\newcommand{\rav}{\ensuremath{\rvert}}
\newcommand{\lv}{\ensuremath{\lVert}}
\newcommand{\rv}{\ensuremath{\rVert}}
\newcommand{\medcap}{\mathbin{\scalebox{1.2}{\ensuremath{\cap}}}}
\newcommand{\medcup}{\mathbin{\scalebox{1.2}{\ensuremath{\cup}}}}
\newcommand{\stimesr}{\ensuremath{{}_s{\times}_r\,}}
\newcommand{\restr}[1]{\ensuremath{\vert_{#1}}}

\hypersetup{pdfauthor={Becky Armstrong, Lisa Orloff Clark, Kristin Courtney, Ying-Fen Lin, Kathryn McCormick, and Jacqui Ramagge}, pdftitle={Twisted Steinberg algebras}}

\title[Twisted Steinberg algebras]{Twisted Steinberg algebras}

\author[Armstrong]{Becky Armstrong}
\author[Clark]{Lisa Orloff Clark}
\author[Courtney]{Kristin Courtney}
\author[Lin]{Ying-Fen Lin}
\author[McCormick]{Kathryn McCormick}
\author[Ramagge]{Jacqui Ramagge}
\address[B.\ Armstrong and K.\ Courtney]{Mathematical Institute, WWU M\"unster, Einsteinstr.\ 62, 48149 M\"unster, GERMANY}
\email{\href{mailto:becky.armstrong@uni-muenster.de}{becky.armstrong}, \href{mailto:kcourtne@uni-muenster.de}{kcourtne@uni-muenster.de}}
\address[L.O.\ Clark]{School of Mathematics and Statistics, Victoria University of Wellington, PO Box 600, Wellington 6140, NEW ZEALAND}
\email{\href{mailto:lisa.clark@vuw.ac.nz}{lisa.clark@vuw.ac.nz}}
\address[Y.-F.\ Lin]{Mathematical Sciences Research Centre, Queen's University Belfast, Belfast, BT7 1NN, UNITED KINGDOM}
\email{\href{mailto:y.lin@qub.ac.uk}{y.lin@qub.ac.uk}}
\address[K.\ McCormick]{Department of Mathematics and Statistics, California State University, Long Beach, CA, UNITED STATES}
\email{\href{mailto:kathryn.mccormick@csulb.edu}{kathryn.mccormick@csulb.edu}}
\address[J.\ Ramagge]{Faculty of Science, Durham University, Durham, DH1 3LE, UNITED KINGDOM}
\email{\href{mailto:jacqui.ramagge@durham.ac.uk}{jacqui.ramagge@durham.ac.uk}}

\date{\today}


\begin{document}

\begin{abstract}
We introduce twisted Steinberg algebras over a commutative unital ring $R$. These generalise Steinberg algebras and are a purely algebraic analogue of Renault's twisted groupoid C*-algebras. In particular, for each ample Hausdorff groupoid $G$ and each locally constant $2$-cocycle $\sigma$ on $G$ taking values in the units $R^\times$, we study the algebra $A_R(G,\sigma)$ consisting of locally constant compactly supported $R$-valued functions on $G$, with convolution and involution ``twisted'' by $\sigma$. We also introduce a ``discretised'' analogue of a twist $\Sigma$ over a Hausdorff \'etale groupoid $G$, and we show that there is a one-to-one correspondence between locally constant $2$-cocycles on $G$ and discrete twists over $G$ admitting a continuous global section. Given a discrete twist $\Sigma$ arising from a locally constant $2$-cocycle $\sigma$ on an ample Hausdorff groupoid $G$, we construct an associated twisted Steinberg algebra $A_R(G;\Sigma)$, and we show that it coincides with $A_R(G,\sigma^{-1})$. Given any discrete field $\F_d$, we prove a graded uniqueness theorem for $A_{\F_d}(G,\sigma)$, and under the additional hypothesis that $G$ is effective, we prove a Cuntz--Krieger uniqueness theorem and show that simplicity of $A_{\F_d}(G,\sigma)$ is equivalent to minimality of $G$.
\end{abstract}

\maketitle
\vspace{-0.5ex}

\section{Introduction}

Steinberg algebras have become a topic of great interest for algebraists and analysts alike since their independent introduction in \cite{Steinberg2010} and \cite{CFST2014}. Before Steinberg algebras were specified by name, they appeared in the details of many constructions of groupoid C*-algebras, such as those in \cite{Exel2008, KP2000, KPRR1997, Paterson1999}. Not only have these algebras provided useful insight into the analytic theory of groupoid C*-algebras, they give rise to interesting examples of $*$-algebras; for example, all Leavitt path algebras and Kumjian--Pask algebras can be realised as Steinberg algebras. Moreover, Steinberg algebras have served as a bridge to facilitate the transfer of concepts and techniques between the algebraic and analytic settings; see \cite{BCFS2014} for one such case.

Thirty years prior to the introduction of Steinberg algebras, Renault \cite{Renault1980} initiated the study of twisted groupoid C*-algebras. These are a generalisation of groupoid C*-algebras in which multiplication and involution are twisted by a $\T$-valued $2$-cocycle on the groupoid. Twisted groupoid C*-algebras have since proved extremely valuable in the study of structural properties for large classes of C*-algebras. In particular, work of Renault \cite{Renault2008}, Tu \cite{Tu1999}, and Barlak and Li \cite{BL2017} has revealed deep connections between twisted groupoid C*-algebras and the UCT problem from the classification program for C*-algebras. For more work on twisted C*-algebras associated to graphs and groupoids, see \cite{AB2018, Boenicke2021, CaHR2013, Gillaspy2015, Kumjian1986, KPS2012, KPS2013, KPS2015, KPS2016, Li2019, SWW2014}.

Given the success of non-twisted Steinberg algebras and the far-reaching significance of C*-algebraic results relating to twisted groupoid C*-algebras, we expect that a purely algebraic analogue of twisted groupoid C*-algebras will supply several versatile classes of $*$-algebras to the literature, as well as a new avenue to approach important problems in C*-algebras.

Throughout, let $R$ be a discrete commutative unital ring with units $R^\times$. Let $\C_d$ denote the ring of complex numbers endowed with the discrete topology, and let $\T_d$ denote the complex unit circle endowed with the discrete topology. In this article, we introduce the notion of a \emph{twisted Steinberg algebra} $A_R(G,\sigma)$ constructed from an ample Hausdorff groupoid $G$ and a locally constant $R^\times$-valued $2$-cocycle $\sigma$ on $G$. Our construction generalises the Steinberg algebra $A_R(G)$, and provides a purely algebraic analogue of the twisted groupoid C*-algebra $C^*(G,\sigma)$ in the case where $R = \C_d$.

In the non-twisted setting, the complex Steinberg algebra and the C*-algebra associated to an ample Hausdorff groupoid $G$ are both built from the convolution algebra $C_c(G)$ of continuous compactly supported complex-valued functions on $G$. The complex Steinberg algebra $A(G)$ is the $*$-subalgebra of $C_c(G)$ consisting of locally constant functions, and the full (or reduced) groupoid C*-algebra $C^*(G)$ (or $C_r^*(G)$) is the closure of $C_c(G)$ with respect to the full (or reduced) C*-norm (see \cite[Chapter~9]{Sims2020}). It turns out (see \cite[Proposition~4.2]{CFST2014}) that $A(G)$ sits densely inside of both the full and the reduced C*-algebras. Therefore, the definition of a twisted complex Steinberg algebra should result in the same inclusions; that is, the twisted complex involutive Steinberg algebra should sit $*$-algebraically and densely inside the twisted groupoid C*-algebra. However, to even make sense of that goal, one must first choose between two methods of constructing a twisted groupoid C*-algebra. The first involves twisting the multiplication on $C^*(G)$ by a continuous $\T$-valued $2$-cocycle $\sigma$, and was introduced by Renault in \cite{Renault1980}.

In \cite{Renault1980}, Renault also observed that the structure of a twisted groupoid C*-algebra with multiplication incorporating a $2$-cocycle $\sigma$ could be realised instead by first twisting the groupoid itself, and then constructing an associated C*-algebra. This is achieved by forming a split groupoid extension
\[
\Go \times \T \hookrightarrow G \times_\sigma \T \twoheadrightarrow G,
\]
where multiplication and inversion on the groupoid $G \times_\sigma \T$ both incorporate a $\T$-valued $2$-cocycle $\sigma$ on $G$, and then defining the twisted groupoid C*-algebra to be the completion of the algebra of $\T$-equivariant functions on $C_c(G \times \T)$ under a C*-norm. A few years later, while developing a C*-analogue of Feldman--Moore theory, Kumjian \cite{Kumjian1986} observed the need for a more general construction arising from a locally split groupoid extension
\[
\Go \times \T\hookrightarrow \Sigma \twoheadrightarrow G,
\]
where $\Sigma$ is not necessarily homeomorphic to $G \times \T$. It turns out that when $G$ is a second-countable ample Hausdorff groupoid, a folklore result (\cref{thm: folklore}) tells us that every twist over $G$ does arise from a $\T$-valued $2$-cocycle on $G$.

Therefore, our first task is to define twisted Steinberg algebras with respect to both notions of a twist. This is the focus of \cref{section: TSGs from 2-cocycles,section: discrete twists}. In \cref{section: TSGs from 2-cocycles}, we define the twisted Steinberg algebra $A_R(G,\sigma)$ by taking an ample Hausdorff groupoid $G$ and twisting the multiplication of the classical Steinberg algebra $A_R(G)$ using a \emph{locally constant} $R^\times$-valued $2$-cocycle $\sigma$ on $G$. We then show that $A_{\C_d}(G,\sigma)$ sits densely inside the twisted groupoid C*-algebra $C^*(G,\sigma)$. In \cref{section: Steinberg algebras from twists}, we give an alternative construction of a twisted Steinberg algebra built using a twist $\Sigma$ over $G$, and then verify that these two definitions of twisted Steinberg algebras agree when the twist over $G$ arises from a $2$-cocycle.

In order to construct a twisted complex Steinberg algebra using a twist over a groupoid, we are forced to first ``discretise" our groupoid extension by replacing the standard topology on $\T$ with the discrete topology. Though this may seem a little artificial to a C*-algebraist, this change is indeed necessary, as we explain in \cref{remarks: why discretise}. (Nonetheless, this should not come as too much of a surprise, given the purely algebraic nature of Steinberg algebras.) Thus, \cref{section: discrete twists over Hausdorff etale groupoids} is dedicated to introducing these discretised groupoid twists and establishing in this setting the aforementioned folklore result for an arbitrary commutative unital ring $R$ (\cref{thm: folklore}). Then in \cref{section: twists and 2-cocycles}, we flesh out the relationships between these twists over groupoids and the cohomology theory of groupoids.

\Cref{section: examples} provides several examples of twisted Steinberg algebras, including a notion of \emph{twisted Kumjian--Pask algebras}. The final two sections of the paper are devoted to proving several important results in Steinberg algebras in the twisted setting, when $R$ is a (discrete) field. In \cref{section: uniqueness and simplicity} we prove a twisted version of the Cuntz--Krieger uniqueness theorem for effective groupoids (\cref{thm: CK uniqueness}), and we show that when $R$ is a discrete field and $G$ is effective, simplicity of $A_R(G,\sigma)$ is equivalent to minimality of $G$ (\cref{thm: simplicity}). Finally, in \cref{section: gradings}, we show that twisted Steinberg algebras inherit a graded structure from the underlying groupoid, and we prove a graded uniqueness theorem for twisted Steinberg algebras (\cref{thm: graded uniqueness}).

\section{Preliminaries} \label{section: prelim}

In this section we introduce some notation, and we recall relevant background information on topological groupoids, continuous $2$-cocycles, and twisted groupoid C*-algebras. Throughout this article, $G$ will always be a locally compact Hausdorff topological groupoid with composable pairs $\Gc \subseteq G \times G$, range and source maps $r, s\colon G \to G$, and unit space $\Go \coloneqq r(G) = s(G)$. We will refer to such groupoids as \hl{Hausdorff groupoids}. For all $\gamma \in G$, we have $r(\gamma) = \gamma \gamma^{-1}$ and $s(\gamma) = \gamma^{-1} \gamma$, where multiplication (or composition) of groupoid elements is evaluated from right to left. We write $G^{(3)}$ for the set of composable triples in $G$; that is,
\[
G^{(3)} \coloneqq \{ (\alpha,\beta,\gamma) \,:\, (\alpha,\beta), \, (\beta,\gamma) \in \Gc \}.
\]
For each $x \in \Go$, we define
\[
G_x \coloneqq s^{-1}(x), \quad G^x \coloneqq r^{-1}(x), \quad \text{and } \quad G_x^x \coloneqq G_x \cap G^x.
\]
For any two subsets $U$ and $V$ of a groupoid $G$, we define
\[
U \stimesr V \coloneqq (U \times V) \cap \Gc, \,\
UV \coloneqq \{\alpha\beta : (\alpha,\beta) \in U \stimesr V \}, \, \text{ and } \
U^{-1} \coloneqq \{\alpha^{-1} : \alpha \in U \}.
\]

We call a subset $B$ of $G$ a \hl{bisection} if there exists an open subset $U$ of $G$ containing $B$ such that $r\restr{U}$ and $s\restr{U}$ are homeomorphisms onto open subsets of $G$. We say that $G$ is \hl{\'etale} if $r$ (or, equivalently, $s$) is a local homeomorphism. If $G$ is \'etale, then $\Go$ is open in $G$, and both $G_x$ and $G^x$ are discrete in the subspace topology for any $x \in \Go$. The range and source maps of an \'etale groupoid are both open, and hence so is the multiplication map (see \cite[Lemma~8.4.11]{Sims2020}\footnote{Although the argument given in \cite[Lemma~8.4.11]{Sims2020} is for second-countable groupoids, it can be adapted to work without the second-countability hypothesis by replacing sequences with nets.}). Moreover, $G$ is \'etale if and only if $G$ has a basis of open bisections (see \cite[Lemma~8.4.9]{Sims2020}). We say that $G$ is \hl{ample} if it has a basis of \emph{compact} open bisections. If $G$ is \'etale, then $G$ is ample if and only if its unit space $\Go$ is totally disconnected (see \cite[Proposition~4.1]{Exel2010}).

If $B$ and $D$ are compact open bisections of an ample Hausdorff groupoid, then $B^{-1}$ and $BD$ are also compact open bisections. In fact, the collection of compact open bisections forms an inverse semigroup under these operations (see \cite[Proposition~2.2.4]{Paterson1999}).

The \hl{isotropy} of a groupoid $G$ is the set
\[
\Iso(G) \coloneqq \{ \gamma \in G : r(\gamma) = s(\gamma) \} = \bigcup_{x \in \Go} G_x^x.
\]
We say that $G$ is \hl{principal} if $\Iso(G) = \Go$, and that $G$ is \hl{effective} if the topological interior of $\Iso(G)$ is equal to $\Go$. We say that $G$ is \hl{topologically principal} if the set $\{ x \in \Go : G_x^x = \{x\} \}$ is dense in $\Go$. Every principal \'etale groupoid is effective and topologically principal. If $G$ is a Hausdorff \'etale groupoid, then $G$ is effective if it is topologically principal, and the converse holds if $G$ is additionally second-countable (see \cite[Lemma~3.1]{BCFS2014}). We will often work with Hausdorff groupoids that are \'etale, ample, or second-countable, but we will explicitly state these assumptions.

Before we describe algebras of functions defined on a groupoid, a few remarks on preliminary point-set topology and notation are in order. Given topological spaces $X$ and $Y$, a function $f\colon X \to Y$ is said to be \hl{locally constant} if every element of $X$ has an open neighbourhood $U$ such that $f\restr{U}$ is constant. Every locally constant function is continuous, and if $Y$ has the discrete topology, then every continuous function $f\colon X \to Y$ is locally constant. Throughout, let $R$ be a commutative unital ring endowed with the discrete topology, and write $R^\times$ for the subgroup of units in $R$.

Given a topological space $X$ and a topological ring $Z$, we define the \hl{support} of a function $f\colon X \to Z$ to be the set
\[
\supp(f) \coloneqq \{ x \in X : f(x) \ne 0 \} = f^{-1}(Z {\setminus} \{0\}).
\]
If $f$ is continuous, then its support is open, and if $f$ is locally constant, then its support is clopen. We say that $f$ is \hl{compactly supported} if $\overline{\supp(f)}$ is compact.

As motivation for our definition of a twisted Steinberg algebra, it will be helpful to briefly recall the construction of groupoid C*-algebras and Steinberg algebras, and to describe the ways in which twisted groupoid C*-algebras have been defined in the literature.

We begin by describing groupoid C*-algebras, which were introduced by Renault in \cite{Renault1980}. In the discussion that follows, it will suffice to restrict our attention to the setting in which the underlying Hausdorff groupoid $G$ is second-countable and \'etale. Although the \'etale assumption is not required, this setting is general enough to include a plethora of examples, including the Cuntz--Krieger algebras of all compactly aligned topological higher-rank graphs (see \cite[Theorem~3.16]{Yeend2007}).

Given a second-countable Hausdorff \'etale groupoid $G$, the \hl{convolution algebra} $C_c(G)$ is the complex $*$-algebra
\[
C_c(G) \coloneqq \{ f\colon G \to \C \,:\, f \text{ is continuous and } \overline{\supp(f)} \text{ is compact} \},
\]
equipped with multiplication given by the \hl{convolution product}
\[
(f * g)(\gamma) \,\coloneqq\, \sum_{\substack{(\alpha,\beta) \in \Gc, \\ \alpha\beta = \gamma}} f(\alpha) \, g(\beta) \,=\, \sum_{\eta \in G^{s(\gamma)}} f(\gamma\eta) \, g(\eta^{-1}),
\]
and involution given by $f^*(\gamma) \coloneqq \overline{f(\gamma^{-1})}$. The \hl{full groupoid C*-algebra} $C^*(G)$ is defined to be the completion of $C_c(G)$ in the \hl{full C*-norm}, and the \hl{reduced groupoid C*-algebra} $C_r^*(G)$ is defined to be the completion of $C_c(G)$ in the \hl{reduced C*-norm} (see \cite[Chapter~9]{Sims2020} for the details).

The first conception of a \emph{twisted} groupoid C*-algebra was also introduced by Renault in \cite{Renault1980}. In this setting, the ``twist'' refers to a continuous $\T$-valued $2$-cocycle on $G$, which is incorporated into the definitions of the multiplication and involution of the convolution algebra $C_c(G)$. Given an arbitrary commutative unital topological ring $R$, a \hl{(continuous) $2$-cocycle} is a continuous function $\sigma\colon \Gc\to R^\times$ that satisfies the \hl{$2$-cocycle identity}:
\[
\sigma(\alpha,\beta) \, \sigma(\alpha\beta,\gamma) = \sigma(\alpha,\beta\gamma) \, \sigma(\beta,\gamma),
\]
for all $(\alpha, \beta, \gamma) \in G^{(3)}$, and is \hl{normalised}, in the sense that
\[
\sigma(r(\gamma),\gamma) = 1 = \sigma(\gamma,s(\gamma)),
\]
for all $\gamma \in G$. We say that the $2$-cocycles $\sigma, \tau\colon \Gc \to R^\times$ are \hl{cohomologous} if there is a continuous function $b\colon G \to R^\times$ satisfying $b(x) = 1$ for all $x \in \Go$, and
\[
\sigma(\alpha,\beta) \, \tau(\alpha,\beta)^{-1} = b(\alpha) \, b(\beta) \, b(\alpha\beta)^{-1}
\]
for all $(\alpha,\beta) \in \Gc$. We may also define $2$-cocycles taking values in a particular subgroup $T$ of $R^\times$, and in this case two $2$-cocycles are cohomologous if there is a function $b$ taking values in $T$ and satisfying the condition above. Cohomology of continuous $2$-cocycles on $G$ is an equivalence relation. The equivalence class of a continuous $2$-cocycle $\sigma$ under this relation is called its \hl{cohomology class}. Note that if we omit the requirement that a $2$-cocycle be normalised, it turns out that every $2$-cocycle that is \emph{not} normalised is cohomologous to one that \emph{is} normalised (see, for example, \cite[Footnote~1, Page~1262]{BaH2014}). Thus, since we show in \cref{lemma: cohomologous twisted Steinberg algebras} that cohomologous $2$-cocycles give isomorphic twisted Steinberg algebras, it makes sense for us to just assume that all $2$-cocycles are normalised.

Given a $2$-cocycle $\sigma\colon \Gc\to \T$, the \hl{twisted convolution algebra} $C_c(G,\sigma)$ is the complex $*$-algebra that is equal as a vector space to $C_c(G)$, but has multiplication given by the \hl{twisted convolution product}
\[
(f *_\sigma g)(\gamma) \coloneqq \sum_{\substack{(\alpha,\beta) \in \Gc, \\ \alpha\beta = \gamma}} \sigma(\alpha,\beta) \, f(\alpha) \, g(\beta) = \sum_{\eta \in G^{s(\gamma)}} \sigma(\gamma\eta,\eta^{-1}) \, f(\gamma\eta) \, g(\eta^{-1}),
\]
and involution given by
\[
f^*(\gamma) \coloneqq \overline{\sigma(\gamma,\gamma^{-1})} \, \overline{f(\gamma^{-1})}.
\]
The $2$-cocycle identity guarantees that the multiplication is associative, and the assumption that the $2$-cocycle is normalised implies that the twist is trivial when either multiplying or applying the involution to functions supported on $\Go$. The \hl{full twisted groupoid C*-algebra} $C^*(G,\sigma)$ is defined to be the completion of $C_c(G,\sigma)$ in the \hl{full C*-norm}, and the \hl{reduced twisted groupoid C*-algebra} $C_r^*(G,\sigma)$ is defined to be the completion of $C_c(G,\sigma)$ in the \hl{reduced C*-norm} (see \cite[Chapter~II.1]{Renault1980} for the details). There is also a $*$-algebra norm on $C_c(G,\sigma)$, called the \hl{$I$-norm}, which is given by
\[
\lv f \rv_{I, \sigma} \coloneqq \max \Big\{\sup_{u \in \Go} \big\{\sum_{\gamma \in G^u} \lav f(\gamma) \rav \big\}, \ \sup_{u \in \Go} \big\{\sum_{\gamma \in G_u } \lav f(\gamma) \rav \big\}\Big\},
\]
for all $f \in C_c(G,\sigma)$. The $I$-norm dominates the full norm on $C_c(G,\sigma)$.

Renault \cite{Renault1980} also introduced an alternative construction of these twisted groupoid C*-algebras involving twisting the groupoid itself, via a split groupoid extension
\[
\Go \times \T \hookrightarrow G \times_\sigma \T \twoheadrightarrow G,
\]
called a \hl{twist} over $G$. In 1986, Kumjian generalised this construction to give twisted groupoid C*-algebras whose twists are not induced by $\T$-valued $2$-cocycles. In particular, the extension $\Sigma$ of $G$ by $\Go \times \T$ need not admit a continuous global section $P\colon G \to \Sigma$. In \cref{section: discrete twists over Hausdorff etale groupoids} we develop a ``discretised'' version of this more general notion of a twist, whose definition is in line with \cite{Brown1982} when $G$ is a discrete group. Since our definition is almost identical to Kumjian's (with the difference being the choice of topology on $\T \le \C^\times$), we refer the reader to \cref{def: twist} for a more precise definition of a twist over a Hausdorff \'etale groupoid. Given a twist
\[
\Go \times \T \hookrightarrow \Sigma \twoheadrightarrow G,
\]
over a Hausdorff \'etale groupoid $G$, one constructs a $*$-algebra by defining an (untwisted) convolution and involution on the subspace of $C_c(\Sigma)$ consisting of $\T$-equivariant functions. Completing this $*$-algebra with respect to the full (or reduced) C*-norm yields the full (or reduced) twisted groupoid C*-algebra $C^*(G,\Sigma)$ (or $C_r^*(G,\Sigma)$). (See \cite{Renault2008} or \cite[Chapter~11]{Sims2020} for more details.)

We conclude this section with the definition of Steinberg algebras, which were originally introduced in \cite{Steinberg2010, CFST2014}, and are a purely algebraic analogue of groupoid C*-algebras. Let $G$ be an \emph{ample} Hausdorff groupoid, and let $1_B$ denote the characteristic function of $B$ from $G$ to $R$. The \hl{Steinberg algebra} associated to $G$ is
\begin{align*}
A_R(G) \,\coloneqq&~\Rspan\{ 1_B\colon G \to R \,:\, B \text{ is a compact open bisection of } G \} \\
=&~\{ f\colon G \to R \,:\, f \text{ is continuous and $\supp(f)$ is compact} \},
\end{align*}
equipped with multiplication given by the \hl{convolution product}
\[
(f * g)(\gamma) \,\coloneqq\, \sum_{\substack{(\alpha,\beta) \in \Gc, \\ \alpha\beta = \gamma}} f(\alpha) \, g(\beta) \,=\, \sum_{\eta \in G^{s(\gamma)}} f(\gamma\eta) \, g(\eta^{-1}).
\]
The \hl{complex} Steinberg algebra $A(G) \coloneqq A_{\C_d}(G)$ is a $*$-algebra with involution given by $f^*(\gamma) \coloneqq \overline{f(\gamma^{-1})}$. It is shown in shown in \cite{Steinberg2010, CFST2014} that $A(G)$ is dense in $C_c(G)$ with respect to both the full and reduced C*-norms.

\section{Twisted Steinberg algebras arising from locally constant \texorpdfstring{$2$-cocycles}{2-cocycles}} \label{section: TSGs from 2-cocycles}

In this section we introduce the twisted Steinberg algebra $A_R(G,\sigma)$ over a discrete commutative unital ring $R$ (or $A(G,\sigma)$ when $R = \C_d$) associated to an ample Hausdorff groupoid $G$ and a continuous $2$-cocycle $\sigma\colon \Gc \to T \le R^\times$. As an $R$-module, the twisted Steinberg algebra is identical to the untwisted version defined in \cref{section: prelim}. That is,
\[
A_R(G,\sigma) \coloneqq \Rspan\{ 1_B\colon G \to R \,:\, B \text{ is a compact open bisection of } G \};
\]
we now emphasise that we are viewing $R$ with the discrete topology.

\begin{lemma} \label{lemma: functions in twisted Steinberg algebras}
Let $G$ be an ample Hausdorff groupoid, and let $C_c(G,R)$ denote the collection of continuous compactly supported functions $f\colon G \to R$. For any continuous $2$-cocycle $\sigma\colon \Gc \to T \le R^\times$, we have the following:
\begin{enumerate}[label=(\alph*)]
\item \label{item: functions are C_c and LC} $A_R(G,\sigma) = C_c(G,R) = \{ f\colon G \to R : f \text{ is locally constant and } \supp(f) \text{ is compact} \}$ as $R$-modules;
\item \label{item: bisection sum} for any $f \in A_R(G,\sigma)$, there exist $\lambda_1, \dotsc, \lambda_n \in R {\setminus} \{0\}$ and mutually disjoint compact open bisections $B_1, \dotsc, B_n \subseteq G$ such that $f = \sum_{i=1}^n \lambda_i 1_{B_i}$.
\end{enumerate}
\end{lemma}

\begin{proof} Part~\cref{item: functions are C_c and LC} follows from the characterisations of the Steinberg algebra $A_R(G)$ given in \cite[Definition~4.1 and Remark~4.2]{Steinberg2010}, because $A_R(G,\sigma)$ and $A_R(G)$ agree as $R$-modules. Similarly, part~\cref{item: bisection sum} follows from \cite[Lemma~2.2]{CE-M2015}.
\end{proof}

From now on, we will use the characterisations of $A_R(G,\sigma)$ as an $R$-module given in \cref{lemma: functions in twisted Steinberg algebras}\cref{item: functions are C_c and LC} interchangeably with the definition.

We equip $A_R(G,\sigma)$ with a multiplication that incorporates the $2$-cocycle $\sigma$ into its definition, thereby distinguishing $A_R(G,\sigma)$ from $A_R(G)$. If we additionally assume that there is an involution $r \mapsto \overline{r}$ on the ring $R$, and that $T$ is a subgroup of $R^\times$ such that $\overline{z} = z^{-1}$ for each $z \in T$ and the $2$-cocycle $\sigma$ is $T$-valued, then we may also define an involution $*$ on $A_R(G,\sigma)$ that will make $A_R(G,\sigma)$ into a $*$-algebra. We call such an involution on $R$ a \hl{$T$-inverse involution}.

\begin{prop} \label{prop: twisted Steinberg algebra}
Let $R$ be a commutative unital ring, let $G$ be an ample Hausdorff groupoid, and let $\sigma\colon \Gc \to R^\times$ be a continuous $2$-cocycle. There is a multiplication (called \hl{(twisted) convolution}) on the $R$-module $A_R(G,\sigma)$, given by
\[
(f *_\sigma g)(\gamma) \coloneqq \sum_{\substack{(\alpha,\beta) \in \Gc, \\ \alpha\beta = \gamma}} \sigma(\alpha,\beta) \, f(\alpha) \, g(\beta) = \sum_{\eta \in G^{s(\gamma)}} \sigma(\gamma\eta,\eta^{-1})\, f(\gamma\eta) \, g(\eta^{-1}),
\]
under which $A_R(G,\sigma)$ is an $R$-algebra. Suppose additionally that $R$ has a $T$-inverse involution $r \mapsto \overline{r}$ for some $T \le R^\times$ and that $\sigma$ is $T$-valued. Then there is an involution on $A_R(G,\sigma)$, given by
\[
f^*(\gamma) \coloneqq \sigma(\gamma,\gamma^{-1})^{-1} \, \overline{f(\gamma^{-1})},
\]
under which $A_R(G,\sigma)$ is a $*$-algebra over $R$. We call $A_R(G,\sigma)$ the \hl{twisted Steinberg algebra} over $R$ associated to the pair $(G,\sigma)$.

The \hl{complex} twisted Steinberg algebra $A(G,\sigma) \coloneqq A_{\C_d}(G,\sigma)$ is a dense $*$-subalgebra of the complex twisted convolution algebra $C_c(G,\sigma)$ with respect to the $I$-norm and the full and reduced C*-norms.
\end{prop}

\begin{remarks} \leavevmode
\begin{enumerate}[label=(\arabic*)]
\item If the $2$-cocycle $\sigma$ is trivial (in the sense that $\sigma\!\left(\Gc\right) = \{1\}$), then $A_R(G,\sigma)$ is identical to $A_R(G)$ as an $R$-algebra.
\item We often write $f * g$ or $fg$ to denote the convolution product $f *_\sigma g$ of functions $f, g \in A_R(G,\sigma)$ if the intended meaning is clear.
\item If $f, g \in A_R(G,\sigma)$, then $\supp(fg) \subseteq \supp(f)\supp(g)$. If $B$ and $D$ are compact open bisections of $G$ such that $\supp(f) = B$ and $\supp(g) = D$, then $\supp(fg) = BD$, and when $A_R(G,\sigma)$ is a $*$-algebra, $\supp(f^*) = B^{-1}$.
\item From the $2$-cocycle identity, one can readily verify that $\sigma(\gamma,\gamma^{-1}) = \sigma(\gamma^{-1},\gamma)$ for any $\gamma \in G$.
\end{enumerate}
\end{remarks}

\begin{proof}[Proof of \cref{prop: twisted Steinberg algebra}]
As $R$-modules, $A_R(G,\sigma) \cong A_R(G)$. We first show that $A_R(G,\sigma)$ is closed under the twisted convolution. Fix $f, g \in A_R(G,\sigma)$. By \cref{lemma: functions in twisted Steinberg algebras}\cref{item: bisection sum}, there exist mutually disjoint compact open bisections $B_1, \dotsc, B_m, C_1, \dotsc, C_n \subseteq G$ and scalars $\lambda_1, \dotsc, \lambda_m, \mu_1, \dotsc, \mu_n \in R {\setminus} \{0\}$ such that
\[
f = \sum_{i=1}^m \lambda_i 1_{B_i} \quad \text{ and } \quad g = \sum_{j=1}^n \mu_j 1_{C_j}.
\]
We claim that $fg \in A_R(G,\sigma)$. Since $G$ is \'etale and $f$ and $g$ have compact support, for each $\gamma \in G$, the set
\[
\left\{ (\alpha,\beta) \in \Gc \,:\, \alpha\beta = \gamma \, \text{ and } \, \sigma(\alpha,\beta) \, f(\alpha) \, g(\beta) \ne 0 \right\}
\]
is finite (see \cite[Proposition~9.1.1]{Sims2020}). Since $\sigma$ is locally constant, we can assume that for all $i \in \{1, \dotsc, m\}$ and $j \in \{1, \dotsc, n\}$, there exists $\nu_{i,j} \in R^\times$ such that $\sigma(\alpha,\beta) = \nu_{i,j}$ for all $(\alpha,\beta) \in (B_i) \stimesr (C_j)$ (because otherwise we can further refine the bisections to ensure that this is true). Thus, for all $\gamma \in G$, we have
\begin{align*}
(f *_\sigma g)(\gamma) &= \sum_{\substack{(\alpha,\beta) \in \Gc, \\ \alpha\beta=\gamma}} \sigma(\alpha,\beta) \, f(\alpha) \, g(\beta) \\
&= \sum_{\substack{(\alpha,\beta) \in \Gc, \\ \alpha\beta=\gamma}} \sigma(\alpha,\beta) \left(\sum_{i=1}^m \lambda_i 1_{B_i}(\alpha)\right) \! \left(\sum_{j=1}^n \mu_j 1_{C_j}(\beta)\right) \\
&= \sum_{\substack{(\alpha,\beta) \in \Gc, \\ \alpha\beta=\gamma}} \sum_{i=1}^m \sum_{j=1}^n \nu_{i,j} \, \lambda_i \, \mu_j \, 1_{B_i}(\alpha) \, 1_{C_j}(\beta) \\
&= \sum_{i=1}^m \sum_{j=1}^n \nu_{i,j} \, \lambda_i \, \mu_j \, 1_{B_i C_j}(\gamma).
\end{align*}
Hence $f *_\sigma g \in A_R(G,\sigma)$. The remainder of the verification that $A_R(G,\sigma)$ is an $R$-algebra is similar to \cite[Proposition~II.1.1]{Renault1980}.

Suppose now that $R$ has a $T$-inverse involution $r \mapsto \overline{r}$ for some $T \le R^\times$. We show that $f^* \in A_R(G,\sigma)$. Since $\sigma$ is locally constant, we can assume that for all $i \in \{1, \dotsc, m\}$, there exists $\kappa_i \in T$ such that $\sigma(\gamma,\gamma^{-1}) = \kappa_i$ for all $\gamma \in B_i$ (because otherwise we can further refine the bisections to ensure that this is true). Thus, for all $\gamma \in G$, we have
\[
f^*(\gamma) \,=\, \sigma(\gamma,\gamma^{-1})^{-1} \, \overline{f(\gamma^{-1})} \,=\, \overline{\sigma(\gamma,\gamma^{-1})} \left(\overline{\sum_{i=1}^m \lambda_i \, 1_{B_i}(\gamma^{-1})}\right) \,=\, \sum_{i=1}^m \, \overline{\kappa_i} \, \overline{\lambda_i} \, 1_{B_i^{-1}}(\gamma).
\]
Hence $f^* \in A_R(G,\sigma)$.

Clearly the proposed involution distributes across sums, and $(\lambda f)^* = \overline{\lambda} f^*$ for all $\lambda \in R$. Fix $\gamma \in G$. Since the involution on $R$ restricts to inversion on $T$, we see that
\[
\left(f^*\right)^*(\gamma) \,=\, \sigma(\gamma,\gamma^{-1})^{-1} \, \overline{f^*(\gamma^{-1})} \,=\, \sigma(\gamma,\gamma^{-1})^{-1} \, \overline{\sigma(\gamma^{-1},\gamma)^{-1}} \, f(\gamma) \,=\, f(\gamma).
\]
Furthermore, we have
\begin{align*}
( f *_\sigma g)^*(\gamma) \,&=\, \sigma(\gamma,\gamma^{-1})^{-1} \, \overline{(f *_\sigma g)(\gamma^{-1})} \\
&=\, \sum_{\substack{(\alpha,\beta) \in \Gc, \\ \alpha\beta = \gamma^{-1} }} \sigma(\gamma,\gamma^{-1})^{-1} \, \sigma(\alpha,\beta)^{-1} \, \overline{f(\alpha)} \, \overline{g(\beta)}, \numberthis \label{eqn: inv of conv}
\end{align*}
and
\begin{align*}
(g^* *_\sigma f^*) (\gamma) \,& =\, \sum_{\substack{(\eta,\zeta) \in \Gc, \\ \eta\zeta = \gamma}} \sigma(\eta,\zeta) \, \sigma(\eta,\eta^{-1})^{-1} \, \overline{g(\eta^{-1})} \, \sigma(\zeta,\zeta^{-1})^{-1} \, \overline{f(\zeta^{-1})} \\
&=\, \sum_{\substack{(\alpha,\beta) \in \Gc, \\ \alpha\beta = \gamma^{-1}}} \sigma(\beta^{-1},\alpha^{-1}) \, \sigma(\beta^{-1},\beta)^{-1} \, \sigma(\alpha^{-1},\alpha)^{-1} \, \overline{f(\alpha)} \, \overline{g(\beta)}. \numberthis \label{eqn: conv of invs}
\end{align*}
Using several applications of the $2$-cocycle identity and that $\sigma$ is normalised, we see that
\begin{align*}
\sigma(\alpha,\beta) \, \sigma(\gamma,\gamma^{-1}) \,&=\, \sigma(\alpha,\beta) \, \sigma(\alpha\beta, \beta^{-1}\alpha^{-1}) \\
&=\, \sigma(\alpha, \beta\beta^{-1}\alpha^{-1}) \, \sigma(\beta, \beta^{-1}\alpha^{-1}) \\
&=\, \sigma(\alpha,\alpha^{-1}) \, \sigma(\beta, \beta^{-1}\alpha^{-1}) \, \sigma(\beta^{-1},\alpha^{-1}) \, \sigma(\beta^{-1},\alpha^{-1})^{-1} \\
&=\, \sigma(\alpha,\alpha^{-1}) \, \sigma(\beta,\beta^{-1}) \, \sigma(\beta\beta^{-1}, \alpha^{-1}) \, \sigma(\beta^{-1},\alpha^{-1})^{-1} \\
&=\, \sigma(\alpha^{-1},\alpha) \, \sigma(\beta^{-1},\beta) \, \sigma(\beta^{-1},\alpha^{-1})^{-1}. \numberthis \label{eqn: 2-cocycle terms}
\end{align*}
Thus, we deduce from \cref{eqn: inv of conv,eqn: conv of invs,eqn: 2-cocycle terms} that $(f *_\sigma g)^* = g^* *_\sigma f^*$, and so $A_R(G,\sigma)$ is a $*$-algebra over $R$.

Finally, since $A_{\C_d}(G,\sigma)$ and $A_{\C_d}(G)$ agree as vector spaces, it follows from \cite[Proposition~2.2.7]{Paterson1999} that $A_{\C_d}(G,\sigma)$ is dense in $C_c(G,\sigma)$ with respect to the $I$-norm, and hence also with respect to the full and reduced C*-norms, since they are both dominated by the $I$-norm.
\end{proof}

Note that we used that $\sigma$ is locally constant in order to show that $A_R(G,\sigma)$ is closed under the twisted convolution and involution.

In the untwisted Steinberg algebra setting, given compact open bisections $B$ and $D$ of $G$, we have $1_B 1_D = 1_{BD}$. This is not the case in the twisted setting, due to the presence of the $2$-cocycle in the convolution formula. Instead, we have the following properties concerning the generators $1_B$ of the twisted Steinberg algebra $A_R(G,\sigma)$.

\begin{lemma} \label{lemma: characteristic function properties}
Let $G$ be an ample Hausdorff groupoid, and let $\sigma\colon \Gc \to R^\times$ be a continuous $2$-cocycle. Suppose that $B$ and $D$ are compact open bisections of $G$.
\begin{enumerate}[label=(\alph*)]
\item \label{item: 1_B 1_D general} For all $(\alpha,\beta) \in B \stimesr D$, we have
\[
(1_B1_D)(\alpha\beta) = \sigma(\alpha,\beta) \, 1_B(\alpha) \, 1_D(\beta) = \sigma(\alpha,\beta) \, 1_{BD}(\alpha\beta) = \sigma(\alpha,\beta).
\]
\item \label{item: 1_B 1_D unit space} If $B \subseteq \Go$ or $D \subseteq \Go$, then $1_B 1_D = 1_{BD}$.
\end{enumerate}
Suppose that $R$ has a $T$-inverse involution $r \mapsto \overline{r}$ for some $T \le R^\times$ and that $\sigma$ is $T$-valued.
\begin{enumerate}[label=(\alph*),resume]
\item \label{item: 1_B involution} For all $\gamma \in G$, we have $1_B^*(\gamma) = \sigma(\gamma,\gamma^{-1})^{-1} \, 1_{B^{-1}}(\gamma)$.
\item \label{item: 1_B range and source} We have $1_B1_B^* = 1_{r(B)}$ and $1_B^*1_B = 1_{s(B)}$.
\item \label{item: 1_B conjugation} We have $1_B1_B^*1_B = 1_B$ and $1_B^*1_B 1_B^* = 1_B^*$.
\end{enumerate}
\end{lemma}

\begin{proof}
\begin{itemize}
\item[\cref{item: 1_B 1_D general}] This follows immediately from the definition of the twisted convolution product because $B$ and $D$ are bisections.
\item[\cref{item: 1_B 1_D unit space}] Suppose that $B \subseteq \Go$ or $D \subseteq \Go$, and fix $\gamma \in G$. If $\gamma \in BD$, then $\gamma = \alpha\beta$ for some pair $(\alpha,\beta) \in B \stimesr D$. Since $\sigma$ is normalised, we have $\sigma(\alpha,\beta) = 1$, and so
\[
(1_B 1_D)(\gamma) = \sigma(\alpha,\beta) \, 1_B(\alpha) \, 1_D(\beta) = 1_B(\alpha) \, 1_D(\beta) = 1_{BD}(\gamma).
\]
If $\gamma \notin BD$, then $(1_B 1_D)(\gamma) = 0 = 1_{BD}(\gamma)$. Thus $1_B 1_D = 1_{BD}$.
\item[\cref{item: 1_B involution}] If $\gamma \in B^{-1}$, then we have
\[
1_B^*(\gamma) = \sigma(\gamma,\gamma^{-1})^{-1} \, \overline{1_B(\gamma^{-1})} = \sigma(\gamma,\gamma^{-1})^{-1} \, 1_{B^{-1}}(\gamma).
\]
If $\gamma \notin B^{-1} = \supp(1_B^*)$, then
\[
1_B^*(\gamma) = 0 = 1_{B^{-1}}(\gamma) = \sigma(\gamma,\gamma^{-1})^{-1} \, 1_{B^{-1}}(\gamma).
\]
\item[\cref{item: 1_B range and source}] We know that $\supp(1_B1_B^*) = B B^{-1} = r(B)$, and for all $\gamma \in B$, we have
\begin{align*}
(1_B1_B^*)(r(\gamma)) &= (1_B1_B^*)(\gamma\gamma^{-1}) \\
&= \sigma(\gamma,\gamma^{-1}) \, 1_B(\gamma) \, 1_B^*(\gamma^{-1}) \\
&= \sigma(\gamma,\gamma^{-1}) \, 1_B(\gamma)\, \sigma(\gamma^{-1},\gamma)^{-1} \, 1_{B^{-1}}(\gamma^{-1}) \qquad \text{(using part~\cref{item: 1_B involution})} \\
&= 1 \\
&= 1_{r(B)}(r(\gamma)).
\end{align*}
Similarly, we have $\supp(1_B^*1_B) = B^{-1} B = s(B)$, and so for all $\gamma \in B$, we have
\begin{align*}
(1_B^*1_B)(s(\gamma)) &= (1_B^*1_B)(\gamma^{-1}\gamma) \\
&= \sigma(\gamma^{-1},\gamma) \, 1_B^*(\gamma^{-1}) \, 1_B(\gamma) \\
&= \sigma(\gamma^{-1},\gamma) \, \sigma(\gamma^{-1},\gamma)^{-1} \, 1_{B^{-1}}(\gamma^{-1}) \, 1_B(\gamma) \qquad \text{(using part~\cref{item: 1_B involution})} \\
&= 1 \\
&= 1_{s(B)}(s(\gamma)).
\end{align*}
\item[\cref{item: 1_B conjugation}] Parts \cref{item: 1_B 1_D unit space} and \cref{item: 1_B range and source} imply that
\[
1_B1_B^*1_B = 1_{r(B)} 1_B = 1_{r(B)B} = 1_B, \quad \text{and } \quad 1_B^*1_B1_B^* = 1_{s(B)} 1_B^*.
\]
Hence $\supp(1_B^*1_B1_B^*) = s(B) B^{-1} = B^{-1}$. For all $\gamma \in B$, we have
\[
(1_B^*1_B1_B^*)(\gamma^{-1}) = \sigma(s(\gamma),\gamma^{-1}) \, 1_{s(B)}(s(\gamma)) \, 1_B^*(\gamma^{-1}) = 1_B^*(\gamma^{-1}),
\]
and so $1_B^*1_B1_B^* = 1_B^*$. \qedhere
\end{itemize}
\end{proof}

The proof of the following result is inspired by the proof of \cite[Proposition~II.1.2]{Renault1980}.

\begin{lemma} \label{lemma: cohomologous twisted Steinberg algebras}
Let $G$ be an ample Hausdorff groupoid, and let $\sigma, \tau\colon \Gc \to T \le R^\times$ be two continuous $2$-cocycles whose cohomology classes coincide. Then $A_R(G,\sigma)$ is isomorphic to $A_R(G,\tau)$. If $R$ has a $T$-inverse involution, then $A_R(G,\sigma)$ is $*$-isomorphic to $A_R(G,\tau)$.
\end{lemma}

\begin{proof}
For this proof, we will use $*$ to denote convolution, in order to distinguish it from the pointwise product.

Since $\sigma$ is cohomologous to $\tau$, there is a continuous function $b\colon G \to T$ satisfying $b(x) = 1$ for all $x \in \Go$, and
\begin{equation} \label{eqn: cohomologous}
\sigma(\alpha,\beta) \, \tau(\alpha,\beta)^{-1} = b(\alpha) \, b(\beta) \, b(\alpha\beta)^{-1},
\end{equation}
for all $(\alpha,\beta) \in \Gc$.

For each $f \in A_R(G,\sigma) = C_c(G,R)$, let $\theta(f)$ denote the pointwise product $bf$. Since $bf\colon G \to R$ is continuous and satisfies $\supp(bf) = \supp(f)$, we have $bf \in C_c(G,R) = A_R(G,\tau)$. We claim that $\theta\colon A_R(G,\sigma) \to A_R(G,\tau)$ is an $R$-algebra isomorphism. It is clear that $\theta$ is $R$-linear. We must show that $\theta$ respects the twisted convolution operation, and that it respects the involution in the case where $R$ has a $T$-inverse involution.

For all $(\alpha,\beta) \in \Gc$, \cref{eqn: cohomologous} implies that
\begin{equation} \label{eqn: b convolution property}
\sigma(\alpha,\beta) \, b(\alpha\beta) = \tau(\alpha,\beta) \, b(\alpha) \, b(\beta).
\end{equation}
Hence, for all $f, g \in A_R(G,\sigma)$ and $\gamma \in G$, we have
\begin{align*}
\big(\theta(f) *_\tau \theta(g)\big)(\gamma) &= \sum_{\substack{(\alpha,\beta) \in \Gc, \\ \alpha\beta = \gamma}} \tau(\alpha,\beta) \, \theta(f)(\alpha) \, \theta(g)(\beta) \\
&= \sum_{\substack{(\alpha,\beta) \in \Gc, \\ \alpha\beta = \gamma}} \tau(\alpha,\beta) \, b(\alpha) f(\alpha) \, b(\beta) \, g(\beta) \\
&= \sum_{\substack{(\alpha,\beta) \in \Gc, \\ \alpha\beta = \gamma}} \sigma(\alpha,\beta) \, b(\alpha\beta) \, f(\alpha) \, g(\beta) \qquad \text{(using \cref{eqn: b convolution property})} \\
&= b(\gamma) \, \sum_{\substack{(\alpha,\beta) \in \Gc, \\ \alpha\beta = \gamma}} \sigma(\alpha,\beta) \, f(\alpha) \, g(\beta) \\
&= \big(b(f *_\sigma g)\big)(\gamma) \\
&= \theta(f *_\sigma g)(\gamma).
\end{align*}
Therefore, $\theta$ is an $R$-algebra homomorphism.

We now show that $\theta$ is a bijection. Define $b^{-1}\colon G \to T$ by $b^{-1}(\gamma) \coloneqq b(\gamma)^{-1}$. For each $h \in A_R(G,\tau)$, we have $b^{-1}h \in A_R(G,\sigma)$, and so $\theta(b^{-1}h) = bb^{-1}h = h$. Hence $\theta$ is surjective. To see that $\theta$ is injective, suppose that $f, g \in A_R(G,\sigma)$ satisfy $\theta(f) = \theta(g)$. Then $f = b^{-1}bf = b^{-1}\theta(f) = b^{-1}\theta(g) = b^{-1}bg = g$. Therefore, $\theta$ is an $R$-algebra isomorphism.

Now suppose that $R$ has a $T$-inverse involution $r \mapsto \overline{r}$. For all $\gamma \in G$, letting $\alpha = \gamma$ and $\beta = \gamma^{-1}$ in \cref{eqn: cohomologous} gives
\[
\sigma(\gamma,\gamma^{-1}) \, \tau(\gamma,\gamma^{-1})^{-1} = b(\gamma) \, b(\gamma^{-1}) \, b(\gamma\gamma^{-1})^{-1} = b(\gamma) \, b(\gamma^{-1}),
\]
and hence
\begin{equation} \label{eqn: b involution property}
b(\gamma) \, \sigma(\gamma,\gamma^{-1})^{-1} = \tau(\gamma,\gamma^{-1})^{-1} \, \overline{b(\gamma^{-1})}.
\end{equation}
Thus, for all $f \in A_R(G,\sigma)$ and $\gamma \in G$, we have
\begin{align*}
\theta(f^*)(\gamma) &= b(\gamma) \, f^*(\gamma) \\
&= b(\gamma) \, \sigma(\gamma,\gamma^{-1})^{-1} \, \overline{f(\gamma^{-1})} \\
&= \tau(\gamma,\gamma^{-1})^{-1} \, \overline{b(\gamma^{-1})} \, \overline{f(\gamma^{-1})} \qquad \text{(using \cref{eqn: b involution property})} \\
&= (bf)^*(\gamma) \\
&= \theta(f)^*(\gamma),
\end{align*}
and so $\theta$ is a $*$-isomorphism.
\end{proof}

\begin{prop}
Let $G$ be an ample Hausdorff groupoid, and let $\sigma\colon \Gc \to T \le R^\times$ be a continuous $2$-cocycle. The set
\[
\{1_B\colon G \to R \,:\, B \text{ is a nonempty compact open subset of } \Go \}
\]
forms a local unit for $A_R(G,\sigma)$. That is, for any finite collection $f_1, \dotsc, f_n \in A_R(G,\sigma)$, there exists a compact open subset $E$ of $\Go$ such that \[
1_E \, f_i = f_i = f_i \, 1_E,
\]
for each $i \in \{1, \dotsc, n\}$.
\end{prop}

\begin{proof}
Since multiplication by $1_E$ for $E \subseteq \Go$ is not affected by the $2$-cocycle, this follows from the analogous non-twisted result \cite[Lemma~2.6]{CEP2018}.
\end{proof}

\section{Twisted Steinberg algebras arising from discrete twists} \label{section: discrete twists}

There is another (often more general) notion of a twisted groupoid C*-algebra which is constructed from a ``twist" over the groupoid itself; that is, from a locally split groupoid extension of a Hausdorff \'etale groupoid $G$ by $\Go \times \T$. In this section, we define a discretised algebraic analogue of this twist and its associated twisted Steinberg algebra. The primary modification is to replace the topological group $\T$ with a discrete subgroup $T$ of $R^\times$. Many of the results in \cref{section: discrete twists over Hausdorff etale groupoids,section: twists and 2-cocycles} have roots or inspiration in Kumjian's study of groupoid C*-algebras built from groupoid extensions in \cite{Kumjian1986}.

The results in \cref{section: discrete twists over Hausdorff etale groupoids,section: twists and 2-cocycles} also hold in the classical setting with the same proofs. If one is interested in $\T$-valued $2$-cocycles, replacing $T$ with $\T$ (endowed with the standard topology) will not change any of the algebraic arguments therein, and the topological arguments carry through mutatis mutandis. As our ultimate focus is algebraic, we present all of our results in terms of $T$.

\subsection{Discrete twists over Hausdorff \'etale groupoids} \label{section: discrete twists over Hausdorff etale groupoids}

The definition of a \emph{twist} over a Hausdorff \'etale groupoid, which we refer to as a \emph{classical twist}, can be found in \cite[Definition~11.1.1]{Sims2020}. The following is our discretised version.

\begin{definition} \label{def: twist}
Let $G$ be a Hausdorff \'etale groupoid, let $R$ be a commutative unital ring, and let $T \le R^\times$. A \hl{discrete twist} by $T$ over $G$ is a sequence
\[
\displaystyle \Go \times T \overset{i} \hookrightarrow \Sigma \overset{q} \twoheadrightarrow G,
\]
where the groupoid $\Go \times T$ is regarded as a trivial group bundle with fibres $T$, $\Sigma$ is a Hausdorff groupoid with $\Sigmao = i\big(\Go \times \{1\}\big)$, and $i$ and $q$ are continuous groupoid homomorphisms that restrict to homeomorphisms of unit spaces, such that the following conditions hold.
\begin{enumerate}[label=(\alph*)]
\item \label{item: exactness} The sequence is exact, in the sense that $i(\{x\} \times T) = q^{-1}(x)$ for every $x \in \Go$, $i$ is injective, and $q$ is a quotient map.\footnote{Although it is not explicitly stated in \cite[Definition~11.1.1]{Sims2020} that the groupoid homomorphism $q\colon \Sigma \to G$ is a quotient map and satisfies $q(i(x,z)) = x$ for every $(x,z) \in \Go \times \T$, it follows from the definition.}
\item \label{item: P_alpha} The groupoid $\Sigma$ is a locally trivial $G$-bundle, in the sense that for each $\alpha \in G$, there is an open bisection $B_\alpha$ of $G$ containing $\alpha$, and a continuous map $P_\alpha\colon B_\alpha \to \Sigma$ such that
\begin{enumerate}[label=(\roman*), ref=\cref{item: P_alpha}(\roman*)]
\item \label[condition]{cond: P_alpha section} $q \circ P_\alpha = \id_{B_\alpha}$; and
\item \label[condition]{cond: P_alpha homeo} the map $(\beta,z) \mapsto i(r(\beta),z) \, P_\alpha(\beta)$ is a homeomorphism from $B_\alpha \times T$ to $q^{-1}(B_\alpha)$.
\end{enumerate}
\item The image of $i$ is central in $\Sigma$, in the sense that $i(r(\varepsilon),z) \, \varepsilon = \varepsilon \, i(s(\varepsilon),z)$ for all $\varepsilon \in \Sigma$ and $z \in T$.
\end{enumerate}
We denote a discrete twist over $G$ either by $(\Sigma, i, q)$, or simply by $\Sigma$. We identify $\Sigmao$ with $\Go$ via $q\restr{\Sigmao}$. A continuous map $P_\alpha\colon B_\alpha \to \Sigma$ is called a \hl{(continuous) local section} if it satisfies \cref{cond: P_alpha section}. A (classical) twist over $G$ has the same definition as above, with the exception that $T$ is replaced by $\T$.
\end{definition}

In brief, we think of a discrete twist by $T$ over $G$ as a locally split extension $\Sigma$ of $G$ by $\Go \times T$, where the image of $\Go \times T$ under $i$ is central in $\Sigma$.

\begin{example}
If $G$ is a discrete group, then a discrete twist over $G$ as defined above is a central extension of $G$.
\end{example}

The following result contains several additional properties of discrete twists, which are consequences of \cref{def: twist}.

\begin{lemma} \label{lemma: twist properties}
Let $G$ be a Hausdorff \'etale groupoid, and let $(\Sigma, i, q)$ be a discrete twist by $T \le R^\times$ over $G$. Then the following conditions hold.
\begin{enumerate}[label=(\alph*)]
\item \label{item: Sigma is etale} The groupoid $\Sigma$ is \'etale.
\item \label{item: i is open} The map $i$ is a homeomorphism onto an open subset of $\Sigma$.
\item \label{item: P_alpha units} The open bisections and continuous local sections in \cref{def: twist}\cref{item: P_alpha} can be chosen so that $P_\alpha(\Go \cap B_\alpha) \subseteq \Sigmao$ for each $\alpha \in G$.
\item \label{item: ample G twist} If $G$ is ample, then the open bisections in \cref{def: twist}\cref{item: P_alpha} can be chosen to be compact.
\end{enumerate}
\end{lemma}

\begin{proof}
For part~\cref{item: Sigma is etale}, we will show that the range map on $\Sigma$ is a local homeomorphism. For this, fix $\varepsilon \in \Sigma$. It suffices to find an open neighbourhood $U_\varepsilon \subseteq \Sigma$ of $\varepsilon$ such that $r\restr{U_\varepsilon}$ is a homeomorphism onto an open subset of $\Sigma$. By \cref{def: twist}\cref{item: P_alpha} there exist an open bisection $B_{q(\varepsilon)}$ of $G$ containing ${q(\varepsilon)}$, and a continuous local section $P_{q(\varepsilon)}\colon B_{q(\varepsilon)} \to \Sigma$, such that the map $\phi_{q(\varepsilon)}\colon B_{q(\varepsilon)} \times T \to q^{-1}(B_{q(\varepsilon)})$ given by $\phi_{q(\varepsilon)}(\beta,z) \coloneqq i(r(\beta),z) \, P_{q(\varepsilon)}(\beta)$ is a homeomorphism. For each $(\beta,z) \in B_{q(\varepsilon)} \times T$, we have $q\big(\phi_{q(\varepsilon)}(\beta,z)\big) = q\big(P_{q(\varepsilon)}(\beta)\big) = \beta$. Since $\varepsilon \in q^{-1}(B_{q(\varepsilon)})$, there is a unique $z_\varepsilon \in T$ such that $\phi_{q(\varepsilon)}(q(\varepsilon),z_\varepsilon) = \varepsilon$. Define $U_\varepsilon \coloneqq \phi_{q(\varepsilon)}\big(B_{q(\varepsilon)} \times \{z_\varepsilon\}\big)$. Then $\varepsilon \in U_\varepsilon$, and since $T$ has the discrete topology and $\phi_{q(\varepsilon)}$ is an open map onto an open subset of $\Sigma$, $U_\varepsilon$ is an open subset of $\Sigma$. Since $q(U_\varepsilon) = B_{q(\varepsilon)}$, we have $r(U_\varepsilon) = (q\restr{\Sigmao})^{-1}\big(r(q(U_\varepsilon))\big) = (q\restr{\Sigmao})^{-1}(r(B_{q(\varepsilon)}))$. Thus $r(U_\varepsilon)$ is open in $\Sigma$, because the range map in $G$ is open and $q\restr{\Sigmao}$ is continuous. To see that $r\restr{U_\varepsilon}$ is injective, suppose that $r(\zeta) = r(\eta)$ for some $\zeta, \eta \in U_\varepsilon$. Then $q(\zeta), q(\eta) \in B_{q(\varepsilon)}$ and $r(q(\zeta)) = q(r(\zeta)) = q(r(\eta)) = r(q(\eta))$, and so $q(\zeta) = q(\eta)$ since $r\restr{B_{q(\varepsilon)}}$ is injective. Thus, we have $\zeta = \phi_{q(\varepsilon)}(q(\zeta),z_\varepsilon) = \phi_{q(\varepsilon)}(q(\eta),z_\varepsilon) = \eta$, and so $r\restr{U_\varepsilon}$ is injective. Therefore, $\Sigma$ is \'etale.

For part~\cref{item: i is open}, note that the image of $i$ is $q^{-1}(\Go)$, which is open in $\Sigma$ because $q$ is continuous and $\Go$ is an open subset of $G$. Since $i$ is injective and continuous by definition, we need only show that $i$ is an open map. Fix $z \in T$ and an open set $U \subseteq \Go$. Then $U$ is open in $G$ because $\Go$ is open in $G$. Since $T$ has the discrete topology, it suffices to show that $i(U \times \{z\})$ is open in $\Sigma$. Fix $x \in U$. By \cref{def: twist}\cref{item: P_alpha} there exist an open bisection $B_x$ of $G$ containing $x$, and a continuous local section $P_x\colon B_x \to \Sigma$, such that the map $\phi_x\colon B_x \times T \to q^{-1}(B_x)$ given by $\phi_x(\gamma,w) \coloneqq i(r(\gamma),w) \, P_x(\gamma)$ is a homeomorphism. Since $\phi_x\restr{(B_x \cap U) \times T}$ is a homeomorphism onto $q^{-1}(B_x \cap U)$, we may assume that $B_x \subseteq U \subseteq \Go$. For each $y \in B_x$, we have $\phi_x(y,1) = i(y,1) \, P_x(y) = P_x(y)$, and so $P_x(B_x) = \phi_x(B_x \times \{1\})$. Since $T$ has the discrete topology and $\phi_x$ is an open map onto an open subset of $\Sigma$, we deduce that $P_x(B_x)$ is an open subset of $\Sigma$. For each $y \in B_x \subseteq \Go$, we have $\phi_x(y,z) = i(y,z) \, P_x(y)$, and hence
\[
i(B_x \times \{z\}) = \phi_x(B_x \times \{z\}) \, P_x(B_x)^{-1}.
\]
Since $\phi_x$ and inversion in $\Sigma$ are homeomorphisms and part~\cref{item: Sigma is etale} implies that multiplication in $\Sigma$ is an open map, we deduce that $i(B_x \times \{z\})$ is an open subset of $\Sigma$. Therefore,
\[
i(U \times \{z\}) = \bigcup_{x \in U} i(B_x \times \{z\})
\]
is an open subset of $\Sigma$, and hence $i$ is a homeomorphism.

For part~\cref{item: P_alpha units}, fix $\alpha \in G$. By \cref{def: twist}\cref{item: P_alpha} there exist an open bisection $D_\alpha$ of $G$ containing $\alpha$, and a continuous local section $S_\alpha\colon D_\alpha \to \Sigma$, such that the map $\phi_{S_\alpha}\colon (\beta, z) \mapsto i(r(\beta),z) \, S_\alpha(\beta)$ is a homeomorphism from $D_\alpha \times T$ to $q^{-1}(D_\alpha)$.

There are two cases to consider. First, suppose that $\alpha \in G {\setminus} \Go$. Define $B_\alpha \coloneqq D_\alpha {\setminus} \Go$ and $P_\alpha \coloneqq S_\alpha\restr{B_\alpha}$. Since $G$ is Hausdorff, $\Go$ is closed, and hence $B_\alpha$ is open. It follows from the definitions of $D_\alpha$ and $S_\alpha$ that $B_\alpha$ is a bisection of $G$ containing $\alpha$, and that $P_\alpha$ is a continuous map satisfying $q \circ P_\alpha = \id_{B_\alpha}$. Since $\Go \cap B_\alpha = \varnothing$, we trivially have $P_\alpha(\Go \cap B_\alpha) \subseteq \Sigmao$. Alternatively, suppose that $\alpha \in \Go$. Define $B_\alpha \coloneqq \Go \cap D_\alpha$ and $P_\alpha \coloneqq (q\restr{\Sigmao})^{-1}\restr{B_\alpha}$. Since $G$ is \'etale, $\Go$ is open, and hence $B_\alpha$ is open. It follows from the definition of $D_\alpha$ that $B_\alpha$ is a bisection of $G$ containing $\alpha$. Since $B_\alpha \subseteq \Go$ and $q$ restricts to a homeomorphism of unit spaces, $P_\alpha$ is a continuous map satisfying $q \circ P_\alpha = \id_{B_\alpha}$ and $P_\alpha(\Go \cap B_\alpha) = (q\restr{\Sigmao})^{-1}(B_\alpha) \subseteq \Sigmao$.

We now show that \cref{cond: P_alpha homeo} of \cref{def: twist} is still satisfied in both cases. Define $\phi_{P_\alpha}(\beta,z) \coloneqq i(r(\beta),z) \, P_\alpha(\beta)$ for all $(\beta,z) \in B_\alpha \times T$. To see that $\phi_{P_\alpha}$ is injective, suppose that $\phi_{P_\alpha}(\beta,z) = \phi_{P_\alpha}(\gamma,w)$ for some $(\beta,z), (\gamma,w) \in B_\alpha \times T$. Since $i(\Go \times T) = q^{-1}(\Go)$ and $q \circ P_\alpha = \id_{B_\alpha}$, we have $\beta = q\big(\phi_{P_\alpha}(\beta,z)\big) = q\big(\phi_{P_\alpha}(\gamma,w)\big) = \gamma$, and hence
\[
i(r(\beta),z) = \phi_{P_\alpha}(\beta,z) \, P_\alpha(\beta)^{-1} = \phi_{P_\alpha}(\gamma,w) \, P_\alpha(\beta)^{-1} = i(r(\gamma),w) = i(r(\beta),w).
\]
It follows from the injectivity of $i$ that $z = w$, and hence $\phi_{P_\alpha}$ is injective. To see that $\phi_{P_\alpha}$ is surjective, fix $\varepsilon \in q^{-1}(B_\alpha)$, and let $\beta \coloneqq q(\varepsilon)$. Then $q\big(\varepsilon \, P_\alpha(\beta)^{-1}\big) = q(\varepsilon) \beta^{-1} = r(\beta)$, and so $\varepsilon \, P_\alpha(\beta)^{-1} \in q^{-1}(r(\beta))$. Hence \cref{def: twist}\cref{item: exactness} implies that there exists $z \in T$ such that $\varepsilon \, P_\alpha(\beta)^{-1} = i(r(\beta),z)$. Thus $\phi_{P_\alpha}(\beta,z) = i(r(\beta),z) \, P_\alpha(\beta) = \varepsilon$, and so $\phi_{P_\alpha}$ is surjective.

If $\alpha \in G {\setminus} \Go$, then $\phi_{P_\alpha} = \phi_{S_\alpha}\restr{B_\alpha \times T}$, and it follows that $\phi_{P_\alpha}$ is open and continuous. If $\alpha \in \Go$, then $B_\alpha \subseteq \Go$, and $\phi_{P_\alpha}(y,z) = i(y,z) \, (q\restr{\Sigmao})^{-1}(y)$ for all $(y,z) \in B_\alpha \times T$. Part~\cref{item: Sigma is etale} implies that multiplication in $\Sigma$ is open, and it follows from the fact that the maps $i$ and $(q\restr{\Sigmao})^{-1}$ and multiplication in $\Sigma$ are all open and continuous that $\phi_{P_\alpha}$ is also open and continuous. Therefore, in either case, $\phi_{P_\alpha}\colon B_\alpha \times T \to q^{-1}(B_\alpha)$ is a homeomorphism.

Part~\cref{item: ample G twist} is immediate, because every ample groupoid has a basis of compact open bisections.
\end{proof}

We define a notion of an isomorphism of discrete twists in an analogous way to the classical version.

\begin{definition} \label{def: isomorphic twists}
Let $G$ be a Hausdorff \'etale groupoid. We say that two discrete twists $(\Sigma, i, q)$ and $(\Sigma', i', q')$ by $T \le R^\times$ over $G$ are \hl{isomorphic} if there exists a groupoid isomorphism\footnote{We say that $\psi\colon \Sigma \to \Sigma'$ is a \hl{groupoid isomorphism} if it is a homeomorphism such that $\psi(\delta\varepsilon) = \psi(\delta) \psi(\varepsilon)$ for all $(\delta,\varepsilon) \in \Sigmac$.} $\psi\colon \Sigma \to \Sigma'$ such that the following diagram commutes.
\[
\begin{tikzcd}
\Go \times T \arrow{r}{i} \arrow[equal]{d} & \Sigma\arrow{r}{q}\arrow{d}{\psi} & G\arrow[equal]{d} \\
\Go \times T \arrow{r}{i'} & \Sigma'\arrow{r}{q'} & G
\end{tikzcd}
\]
\end{definition}

It is natural to ask whether there is a correspondence between discrete twists over a groupoid and locally constant $2$-cocycles which can be used to ``twist'' the multiplication in Steinberg algebras, given the shared terminology. As one familiar with the literature would expect, we can readily build a twist over a Hausdorff \'etale groupoid from a locally constant $2$-cocycle. To demonstrate this, we adapt the construction outlined in \cite[Example~11.1.5]{Sims2020} to the setting where the continuous $2$-cocycle maps into a discrete group $T \le R^\times$ (rather than $\T$).

\begin{example} \label{eg: twisted groupoid from cocycle}
Let $G$ be a Hausdorff \'etale groupoid, and let $\sigma\colon \Gc \to T \le R^\times$ be a continuous $2$-cocycle. Let $G \times_\sigma T$ be the set $G \times T$ endowed with the product topology, with multiplication given by
\[
(\alpha,z)(\beta,w) \coloneqq (\alpha\beta, \, \sigma(\alpha,\beta) \, zw),
\]
and inversion given by
\[
(\alpha,z)^{-1} \coloneqq (\alpha^{-1}, \, \sigma(\alpha,\alpha^{-1})^{-1} \, z^{-1}) = (\alpha^{-1}, \, \sigma(\alpha^{-1},\alpha)^{-1} \, z^{-1}),
\]
for all $(\alpha,\beta) \in \Gc$ and $z, w \in T$. Then $G \times_\sigma T$ is a Hausdorff groupoid. In fact, unlike in the classical setting, $G$ being \'etale implies that $G \times_\sigma T$ is \'etale, because for each $z \in T$ and bisection $U$ of $G$, $r\restr{U \times \{z\}}$ is a homeomorphism onto $r(U) \times \{1\}$. Define $i\colon \Go \times T \to G \times_\sigma T$ by $i(x,z) \coloneqq (x,z)$, and $q\colon G \times_\sigma T\to G$ by $q(\gamma,z) \coloneqq \gamma$. Then $q$ is easily verified to be a quotient map, and since $\sigma$ is normalised, $i$ is an injective groupoid homomorphism. Just as in \cite[Example~11.1.5]{Sims2020}, it is routine to then check that $(G \times_\sigma T, i,q)$ is a discrete twist by $T$ over $G$.
\end{example}

\Cref{eg: twisted groupoid from cocycle} shows that any locally constant $2$-cocycle on a Hausdorff \'etale groupoid $G$ gives rise to a discrete twist over $G$; the converse is true when $G$ is additionally second-countable and ample. The proof of this fact and its consequences will be the focus of the remainder of this subsection.

Before we proceed, we need two technical results regarding the left and right group actions of $T$ on $\Sigma$ that are induced by the map $i\colon \Go \times T \to \Sigma$. Identifying $\Sigmao$ with $\Go$, these actions are given by
\[
z \cdot \varepsilon \coloneqq i(r(\varepsilon),z) \, \varepsilon \quad \text{ and } \quad \varepsilon \cdot z \coloneqq \varepsilon \, i(s(\varepsilon),z),
\]
for each $z \in T$ and $\varepsilon \in \Sigma$. Since the image of $i$ is central in $\Sigma$, we have $z \cdot \varepsilon = \varepsilon \cdot z$, and $(z \cdot \varepsilon)(w \cdot \delta) = (zw) \cdot (\varepsilon\delta)$ for all $(\varepsilon,\delta) \in \Sigmac$ and $z, w \in T$.

\begin{lemma} \label{lemma: twist isomorphisms respect T-action}
Let $G$ be a Hausdorff \'etale groupoid. Suppose that $(\Sigma_1, i_1, q_1)$ and $(\Sigma_2, i_2, q_2)$ are discrete twists by $T \le R^\times$ over $G$, and that $\psi\colon \Sigma_1 \to \Sigma_2$ is an isomorphism of twists, as defined in \cref{def: isomorphic twists}. Then $\psi$ respects the action of $T$, in the sense that $\psi(z \cdot \varepsilon) = z \cdot \psi(\varepsilon)$ for all $z \in T$ and $\varepsilon \in \Sigma_1$.
\end{lemma}

\begin{proof}
Since $\psi\colon \Sigma_1 \to \Sigma_2$ is an isomorphism of twists, we have $i_2 = \psi \circ i_1$. Thus, for all $z \in T$ and $\varepsilon \in \Sigma_1$, we have $\psi(z \cdot \varepsilon) = \psi\big(i_1(r(\varepsilon),z) \, \varepsilon\big) = i_2(r(\varepsilon),z) \, \psi(\varepsilon) = z \cdot \psi(\varepsilon)$.
\end{proof}

The following result is inspired by \cite[Lemma~11.1.3]{Sims2020}.

\begin{lemma} \label{lemma: q and action}
Let $G$ be a Hausdorff \'etale groupoid, and let $(\Sigma, i, q)$ be a discrete twist by $T \le R^\times$ over $G$. Suppose that $\delta, \varepsilon \in \Sigma$ satisfy $q(\delta) = q(\varepsilon)$. Then $r(\delta) = r(\varepsilon)$, and there is a unique $z \in T$ such that $\varepsilon = z \cdot \delta$.
\end{lemma}

\begin{proof}
Fix $\delta, \varepsilon \in \Sigma$ such that $q(\delta) = q(\varepsilon)$. Then $q(r(\delta)) = r(q(\delta)) = r(q(\varepsilon)) = q(r(\varepsilon))$, and hence $r(\delta) = r(\varepsilon)$, because $q$ restricts to a homeomorphism of unit spaces. Thus $q(\varepsilon\delta^{-1}) = q(\varepsilon) q(\varepsilon)^{-1} = r(q(\varepsilon)) \in \Go$, and hence there is a unique element $z \in T$ such that $\varepsilon\delta^{-1} = i\big(r(q(\varepsilon)),z\big)$. By identifying $\Sigmao$ with $\Go$, we obtain $\varepsilon = i(r(\varepsilon),z) \, \delta = z \cdot \delta$.
\end{proof}

Notice that in the case where $\Sigma$ is the twist $G \times_\sigma T$ described in \cref{eg: twisted groupoid from cocycle}, we can check \cref{lemma: q and action} directly. Identifying $\Sigmao = \Go \times \{1\}$ with $\Go$, we have
\[
z \cdot (\alpha,w) = i(r(\alpha),z) (\alpha,w) = (r(\alpha),z)(\alpha,w) = (\alpha,zw),
\]
for all $z \in T$ and $(\alpha,w) \in \Sigma$. If $q(\delta) = q(\varepsilon)$ for some $\delta, \varepsilon \in \Sigma$, then $\delta = (\alpha,w_1)$ and $\varepsilon = (\alpha,w_2)$ for some $\alpha \in G$ and unique $w_1, w_2 \in T$. Since $T$ is a group, there is a unique $z \in T$ such that $z w_1 = w_2$, and hence $z \cdot \delta = (\alpha,zw_1) = \varepsilon$.

Our key tool in what follows will be a \hl{(continuous) global section}; that is, a continuous map $P\colon G\to \Sigma$ satisfying $q \circ P = \id_G$ and $P(\Go) \subseteq \Sigmao = i(\Go\times\{1\})$. Our next result shows that every discrete twist admitting a continuous global section is isomorphic to a discrete twist coming from a locally constant $2$-cocycle, as described in \cref{eg: twisted groupoid from cocycle}. Parts of this result are inspired by the analogous classical versions in \cite[Section~4]{Kumjian1986} and \cite[Chapter~11]{Sims2020}.

\begin{prop} \label{prop: section induces cocycle}
Let $G$ be a Hausdorff \'etale groupoid, and let $(\Sigma, i, q)$ be a discrete twist by $T \le R^\times$ over $G$. Suppose that $\Sigma$ is \hl{topologically trivial}, in the sense that it admits a continuous global section $P\colon G \to \Sigma$. Then the following conditions hold.
\begin{enumerate}[label=(\alph*)]
\item \label{item: P induces sigma} The continuous global section $P$ preserves composability, and induces a continuous $2$-cocycle $\sigma\colon \Gc \to T$ satisfying
\[
P(\alpha) P(\beta) P(\alpha\beta)^{-1} = i\big(r(\alpha), \, \sigma(\alpha,\beta)\big),
\]
for all $(\alpha,\beta) \in \Gc$.
\item \label{item: P properties} For all $(\alpha,\beta) \in \Gc$, we have
\[
P(\alpha) P(\beta) = \sigma(\alpha,\beta) \cdot P(\alpha\beta) \quad \text{ and } \quad P(\alpha)^{-1} = \sigma(\alpha,\alpha^{-1})^{-1} \cdot P(\alpha^{-1}).
\]
\item \label{item: P induces phi_P} Let $(G \times_\sigma T, i_\sigma, q_\sigma)$ be the discrete twist from \cref{eg: twisted groupoid from cocycle}. The map $\phi_P\colon G \times_\sigma T \to \Sigma$ defined by $\phi_P(\alpha,z) \coloneqq z \cdot P(\alpha)$ gives an isomorphism of the twists $G \times_\sigma T$ and $\Sigma$.
\end{enumerate}
\end{prop}

\begin{proof}
For \cref{item: P induces sigma}, fix $(\alpha,\beta) \in \Gc$. Since $q \circ P = \id_G$ and $q$ is a groupoid homomorphism that restricts to a homeomorphism of unit spaces, we have
\[
q(s(P(\alpha))) = s(q(P(\alpha))) = s(\alpha) = r(\beta) = r(q(P(\beta))) = q(r(P(\beta))),
\]
and hence $(P(\alpha), P(\beta)) \in \Sigmac$. We have
\[
q\big(P(\alpha) P(\beta) P(\alpha\beta) ^{-1}\big) = q(P(\alpha)) \, q(P(\beta)) \, q(P(\alpha\beta))^{-1} = \alpha\beta (\alpha\beta)^{-1} = r(\alpha) = q\big(P(r(\alpha))\big),
\]
and so \cref{lemma: q and action} implies that there is a unique element $\sigma(\alpha,\beta) \in T$ such that
\begin{equation} \label{eqn: action of cocycle}
P(\alpha) P(\beta) P(\alpha\beta)^{-1} = \sigma(\alpha,\beta) \cdot P(r(\alpha)) = i\big(r(\alpha), \, \sigma(\alpha,\beta)\big).
\end{equation}
Therefore, $\sigma(\alpha,\beta) = (\pi_2 \circ i^{-1})\big(P(\alpha) P(\beta) P(\alpha\beta)^{-1}\big)$, where $\pi_2$ is the projection of $\Go \times T$ onto the second coordinate. Since $i$ is a homeomorphism onto its image by \cref{lemma: twist properties}\cref{item: i is open}, we deduce that $\sigma$ is continuous because it is a composition of continuous functions.

To check that $\sigma$ satisfies the $2$-cocycle identity, we fix $(\alpha,\beta,\gamma) \in G^{(3)}$ and show that
\[
\sigma(\beta,\gamma) = \sigma(\alpha,\beta) \, \sigma(\alpha\beta,\gamma) \, \sigma(\alpha,\beta\gamma)^{-1}.
\]
Since the image of $i$ is central in $\Sigma$, we have
\begin{equation} \label{eqn: P(alpha) commutes with i term}
i\big(r(\alpha), \, \sigma(\beta,\gamma)\big) \, P(\alpha) = P(\alpha) \, i\big(s(\alpha), \, \sigma(\beta,\gamma)\big) = P(\alpha) \, i\big(r(\beta), \, \sigma(\beta,\gamma)\big).
\end{equation}
Using \cref{eqn: P(alpha) commutes with i term} for the first equality below and \cref{eqn: action of cocycle} for the second and fourth equalities, we obtain
\begin{align*}
i\big(r(\alpha), \, \sigma(\beta,\gamma)\big) &= P(\alpha) \, i\big(r(\beta), \, \sigma(\beta,\gamma)\big) \, P(\alpha)^{-1} \\
&= P(\alpha) P(\beta) P(\gamma) P(\beta\gamma)^{-1} P(\alpha)^{-1} \\
&= \big(P(\alpha) P(\beta) P(\alpha\beta)^{-1}\big) \big(P(\alpha\beta) P(\gamma) P(\alpha\beta\gamma)^{-1}\big) \big(P(\alpha\beta\gamma) P(\beta\gamma)^{-1} P(\alpha)^{-1}\big) \\
&= i\big(r(\alpha), \, \sigma(\alpha,\beta)\big) \, i\big(r(\alpha\beta), \, \sigma(\alpha\beta,\gamma)\big) \, i\big(r(\alpha), \, \sigma(\alpha, \beta\gamma)\big)^{-1} \\
&= i\big(r(\alpha), \, \sigma(\alpha,\beta) \, \sigma(\alpha\beta,\gamma) \, \sigma(\alpha,\beta\gamma)^{-1}\big).
\end{align*}
Thus, by the injectivity of $i$, we deduce that $\sigma$ satisfies the $2$-cocycle identity.

To see that $\sigma$ is normalised, first note that for all $\alpha$ in $G$,
\begin{equation} \label{eqn: showing that sigma is normalised}
q\big(i\big(r(\alpha), \, \sigma(r(\alpha),\alpha)\big)\big) = q\big(i\big(r(\alpha), \, \sigma(\alpha,s(\alpha))\big)\big) = q\big(i(r(\alpha),1)\big) = r(\alpha),
\end{equation}
and $i(r(\alpha),1) \in \Sigmao$. Moreover, by \cref{eqn: action of cocycle}, we have
\[
i\big(r(\alpha), \, \sigma(r(\alpha),\alpha)\big) = P(r(\alpha)) P(\alpha) P(r(\alpha) \alpha)^{-1} = P(r(\alpha)) \in \Sigmao,
\]
and, since $P(s(\alpha)) \in \Sigmao$,
\[
i\big(r(\alpha), \, \sigma(\alpha,s(\alpha))\big) = P(\alpha) P(s(\alpha)) P(\alpha s(\alpha))^{-1} = P(\alpha) P(\alpha)^{-1} = r(P(\alpha)) \in \Sigmao.
\]
Since $q$ restricts to a homeomorphism of unit spaces and $i$ is injective, we deduce from \cref{eqn: showing that sigma is normalised} that for all $\alpha \in G$,
\[
\sigma(r(\alpha),\alpha) = \sigma(\alpha,s(\alpha)) = 1.
\]

For \cref{item: P properties}, fix $(\alpha,\beta) \in \Gc$. Then \cref{eqn: action of cocycle} implies that
\[
P(\alpha) P(\beta) = i\big(r(\alpha\beta), \, \sigma(\alpha,\beta)\big) \, P(\alpha\beta) = \sigma(\alpha,\beta) \cdot P(\alpha\beta),
\]
and also that
\[
P(\alpha) P(\alpha^{-1}) P(\alpha\alpha^{-1})^{-1} = i\big(r(\alpha), \, \sigma(\alpha,\alpha^{-1})\big).
\]
Since $P(\alpha\alpha^{-1})^{-1} = P(r(\alpha)) \in \Sigmao$, we deduce that
\[
P(\alpha)^{-1} = P(\alpha^{-1}) \, i\big(r(\alpha), \, \sigma(\alpha,\alpha^{-1})\big)^{-1} = P(\alpha^{-1}) \cdot \sigma(\alpha,\alpha^{-1})^{-1} = \sigma(\alpha,\alpha^{-1})^{-1} \cdot P(\alpha^{-1}).
\]

For \cref{item: P induces phi_P}, define $\phi_P\colon G \times_\sigma T \to \Sigma$ by $\phi_P(\alpha,z) \coloneqq z \cdot P(\alpha) = i(r(\alpha),z) \, P(\alpha)$. Then $\phi_P$ is continuous, because it is the pointwise product of the continuous maps $i \circ (r \times \id)$ and $P \circ \pi_1$ from $G \times_\sigma T$ to $\Sigma$, where $\pi_1$ is the projection of $G \times_\sigma T$ onto the first coordinate. To see that $\phi_P$ is injective, suppose that $(\alpha,z), (\beta,w) \in \Gc$ satisfy $\phi_P(\alpha,z) = \phi_P(\beta,w)$. Then
\[
\alpha = q(i(r(\alpha),z)) \, q(P(\alpha)) = q(\phi_P(\alpha,z)) = q(\phi_P(\beta,w)) = q(i(r(\beta),w)) \, q(P(\beta)) = \beta.
\]
Therefore,
\[
i(r(\alpha),z) = \phi_P(\alpha,z) \, P(\alpha)^{-1} = \phi_P(\beta,w) \, P(\beta)^{-1} = i(r(\beta),w) = i(r(\alpha),w),
\]
and since $i$ is injective, we have $z = w$. Thus $\phi_P$ is injective. To see that $\phi_P$ is surjective, fix $\varepsilon \in \Sigma$. Since $q(\varepsilon) = q\big(P(q(\varepsilon))\big)$, \cref{lemma: q and action} implies that there exists a unique element $z_\varepsilon \in T$ such that
\[
\phi_P(q(\varepsilon),z_\varepsilon) = z_\varepsilon \cdot P(q(\varepsilon)) = i(r(\varepsilon),z_\varepsilon) \, P(q(\varepsilon)) = \varepsilon.
\]
Thus $\phi_P$ is surjective, and we have $z_\varepsilon = \pi_2\big(i^{-1}\big(\varepsilon \, P(q(\varepsilon))^{-1}\big)\big)$, where $\pi_2$ is the projection of $\Go \times T$ onto the second coordinate. Since $\phi_P^{-1}(\varepsilon) = (q(\varepsilon),z_\varepsilon)$ and \cref{lemma: twist properties}\cref{item: i is open} implies that $i^{-1}$ is continuous on the image of $i$, we deduce that $\phi_P^{-1}$ is continuous, because it is a composition of continuous maps. Hence $\phi_P$ is a homeomorphism.

To see that $\phi_P$ is also a groupoid homomorphism, fix $(\alpha,\beta) \in \Gc$ and $z, w \in T$. Then, using part~\cref{item: P properties} for the third equality, we have
\begin{align*}
\phi_P(\alpha,z) \, \phi_P(\beta,w) &= (z \cdot P(\alpha)) (w \cdot P(\beta)) \\
&= (zw) \cdot (P(\alpha) P(\beta)) \\
&= (zw) \cdot \big(\sigma(\alpha,\beta) \cdot P(\alpha\beta)\big) \\
&= \big(\sigma(\alpha,\beta) zw\big) \cdot P(\alpha\beta) \\
&= \phi_P\big(\alpha\beta, \, \sigma(\alpha,\beta) zw\big) \\
&= \phi_P\big((\alpha,z)(\beta,w)\big).
\end{align*}
Hence $\phi_P$ is a groupoid isomorphism.

We conclude by showing that $\phi_P\circ i_\sigma = i$ and $q \circ \phi_P = q_\sigma$. Recall from \cref{eg: twisted groupoid from cocycle} that $i_\sigma\colon \Go \times T \to G \times_\sigma T$ is the inclusion map and $q_\sigma\colon G \times_\sigma T \to G$ is the projection onto the first coordinate. Fix $\alpha \in G$ and $w \in T$. Since $P(r(\alpha)) \in \Sigmao$, we have
\[
(\phi_P \circ i_\sigma)(r(\alpha),w) = \phi_P(r(\alpha),w) = i(r(\alpha),w) \, P(r(\alpha)) = i(r(\alpha),w),
\]
and
\[
(q \circ \phi_P)(\alpha,w) = q\big(i(r(\alpha),w) \, P(\alpha)\big) = r(\alpha) \alpha = \alpha = q_\sigma(\alpha,w).
\]
Therefore, $\Sigma$ and $G \times_\sigma T$ are isomorphic as twists over $G$.
\end{proof}

As one might expect, all discrete twists constructed from locally constant $2$-cocycles (as in \cref{eg: twisted groupoid from cocycle}) are topologically trivial, as we now prove.

\begin{lemma} \label{lemma: trivial section induces sigma}
Let $G$ be a Hausdorff \'etale groupoid, and let $\sigma\colon \Gc \to T \le R^\times$ be a continuous $2$-cocycle. The twist $(G \times_\sigma T, i, q)$ described in \cref{eg: twisted groupoid from cocycle} is topologically trivial, and the map $S\colon \gamma \mapsto (\gamma,1)$ is a continuous global section from $G$ to $G \times_\sigma T$ that induces $\sigma$.
\end{lemma}

\begin{proof}
It is clear that $S\colon G \to G \times_\sigma T$ is a continuous global section, and so $G \times_\sigma T$ is topologically trivial. By \cref{prop: section induces cocycle}\cref{item: P induces sigma}, $S$ induces a continuous $2$-cocycle $\omega\colon \Gc \to T$ satisfying $S(\alpha) S(\beta) S(\alpha\beta)^{-1} = i\big(r(\alpha), \, \omega(\alpha,\beta)\big) = (r(\alpha), \, \omega(\alpha,\beta))$, for all $(\alpha,\beta) \in \Gc$. To see that $S$ induces $\sigma$, fix $(\alpha,\beta) \in \Gc$. Then
\begin{align*}
(r(\alpha), \, \omega(\alpha,\beta)) &= S(\alpha) S(\beta) S(\alpha\beta)^{-1} \\
&= (\alpha,1) (\beta,1) (\alpha\beta,1)^{-1} \\
&= (\alpha\beta, \, \sigma(\alpha,\beta)) \, \big((\alpha\beta)^{-1}, \, \sigma(\alpha\beta, (\alpha\beta)^{-1})^{-1}\big) \\
&= \big(r(\alpha\beta), \, \sigma(\alpha\beta, (\alpha\beta)^{-1}) \, \sigma(\alpha,\beta) \, \sigma(\alpha\beta, (\alpha\beta)^{-1})^{-1}\big) \\
&= (r(\alpha), \, \sigma(\alpha,\beta)).
\end{align*}
Therefore, $\sigma = \omega$, and so $S$ induces $\sigma$.
\end{proof}

Together, \cref{prop: section induces cocycle,lemma: trivial section induces sigma} give us a one-to-one correspondence between discrete twists over a Hausdorff \'etale groupoid $G$ that admit a continuous global section and discrete twists over $G$ arising from locally constant $2$-cocycles on $G$.

As we shall see in \cref{thm: folklore}, it turns out that all discrete twists over a second-countable ample Hausdorff groupoid $G$ admit a continuous global section. We are grateful to Elizabeth Gillaspy for alerting us to this folklore fact for $T = \T_d$, citing conversations with Alex Kumjian. Because we know of no proofs in the literature, we give a detailed proof here in the discrete setting.

\begin{thm} \label{thm: folklore}
Let $G$ be a second-countable ample Hausdorff groupoid, and let $(\Sigma, i, q)$ be a discrete twist by $T \le R^\times$ over $G$. Then $\Sigma$ is topologically trivial.
\end{thm}

In order to prove \cref{thm: folklore}, we need the following lemma.

\begin{lemma} \label{lemma: disjointify}
Let $G$ be a second-countable ample Hausdorff groupoid, and suppose that $\UU$ is an open cover of $G$. Then $\UU$ has a countable refinement $\{B_j\}_{j=1}^\infty$ of mutually disjoint compact open bisections that form a cover of $G$.
\end{lemma}

\begin{proof}
Let $\UU$ be an open cover of $G$. By possibly passing to a refinement, we may assume that $\UU$ consists of compact open bisections. Since $G$ is second-countable, it is Lindel\"of, and so we may assume that $\UU = \{D_j\}_{j=1}^\infty$, where each $D_j$ is a compact open bisection of $G$. Define $B_1 \coloneqq D_1$, and for each $n \ge 2$, define $B_n \coloneqq D_n {\setminus} \medcup_{i=1}^{n-1} B_i$. Then each $B_j$ is a compact open bisection contained in $D_j$, and $\{B_j\}_{i=j}^\infty$ forms a disjoint cover of $G$.
\end{proof}

\begin{proof}[Proof of \cref{thm: folklore}]
Recall from \cref{def: twist}\cref{item: P_alpha} that for each $\alpha \in G$, there exists an open bisection $D_\alpha \subseteq G$ containing $\alpha$, and a continuous local section $P_\alpha\colon D_\alpha \to \Sigma$ such that the map $\phi_\alpha\colon D_\alpha \times T \to q^{-1}(D_\alpha)$ given by $\phi_\alpha(\beta,z) \coloneqq i(r(\beta),z) \, P_\alpha(\beta) = z \cdot P_\alpha(\beta)$ is a homeomorphism. Since $G$ is ample, we may assume that each $D_\alpha$ is compact, by \cref{lemma: twist properties}\cref{item: ample G twist}. By \cref{lemma: twist properties}\cref{item: P_alpha units}, we may assume that $P_\alpha(\Go \cap D_\alpha) \subseteq \Sigmao$ for each $\alpha \in G$. By \cref{lemma: disjointify}, $\{D_\alpha\}_{\alpha \in G}$ has a countable refinement $\{B_j\}_{j=1}^\infty$ consisting of mutually disjoint compact open bisections that form a cover of $G$. For each $j \ge 1$, choose $\alpha_j \in G$ such that $B_j \subseteq D_{\alpha_j}$, and define $P_j \coloneqq P_{\alpha_j}\restr{B_j}$. For each $\beta \in G$, there is a unique $j_{\beta} \ge 1$ such that $\beta \in B_{j_\beta}$, and hence the map $P\colon G \to \Sigma$ given by $P(\beta) \coloneqq P_{j_\beta}(\beta)$ is well-defined. Since $q(P(\beta)) = q(P_{j_\beta}(\beta)) = \beta = \id_G(\beta)$ for all $\beta \in G$, and $P_j(\Go \cap B_j) \subseteq \Sigmao$ for each $j \ge 1$, $P$ is a global section. To see that $P$ is continuous, let $U$ be an open subset of $\Sigma$. Then $P^{-1}(U) = \medcup_{j=1}^\infty \, P_j^{-1}(U) = \medcup_{j=1}^\infty \, \big(P_{\alpha_j}^{-1}(U) \cap B_j\big)$. Since each $P_{\alpha_j}$ is continuous and each $B_j$ is open, $P^{-1}(U)$ is open in $G$. Hence $P$ is a continuous global section, and $\Sigma$ is topologically trivial.
\end{proof}

\subsection{Twists and \texorpdfstring{$2$-cocycles}{2-cocycles}} \label{section: twists and 2-cocycles}

In this section we restrict our attention to discrete twists arising from locally constant $2$-cocycles, and we investigate the relationships between such twists. In particular, we prove the following theorem.

\begin{thm} \label{thm: cohomology}
Let $G$ be a Hausdorff \'etale groupoid, and let $\sigma, \tau\colon \Gc \to T \le R^\times$ be continuous $2$-cocycles. The following are equivalent:
\begin{enumerate}[label=(\arabic*)]
\item \label{item: thm isomorphic} $G \times_\sigma T \cong G \times_\tau T$;
\item \label{item: thm cohomologous} $\sigma$ is cohomologous to $\tau$; and
\item \label{item: thm induced} $\sigma$ is induced by a continuous global section $P\colon G \to G \times_\tau T$.
\end{enumerate}
\end{thm}

We will split the proof of this theorem into three lemmas. This proof has notable overlap with \cite[Section~4]{Kumjian1986} for the case where $R = \C_d$ and $T = \T_d$, particularly for the equivalence of \cref{item: thm cohomologous} and \cref{item: thm induced}. However, the two formulations are sufficiently different to warrant independent treatment here.

The following lemma expands on an argument given in \cite[Remark~11.1.6]{Sims2020} showing that the cohomology class of a continuous $2$-cocycle $\sigma\colon \Gc \to T \le R^\times$ can always be recovered from the discrete twist $G \times_\sigma T$.

\begin{lemma} \label{lemma: section implies cohomologous}
Let $G$ be a Hausdorff \'etale groupoid, and let $\tau\colon \Gc \to T \le R^\times$ be a continuous $2$-cocycle. Suppose that $P\colon G \to G \times_\tau T$ is a continuous global section, and that $\sigma\colon \Gc \to T$ is the induced continuous $2$-cocycle satisfying
\[
i\big(r(\alpha), \, \sigma(\alpha,\beta)\big) = P(\alpha) P(\beta) P(\alpha\beta)^{-1}
\]
for all $(\alpha,\beta) \in \Gc$, as in \cref{prop: section induces cocycle}. Then $\sigma$ is cohomologous to $\tau$.
\end{lemma}

\begin{proof}
To see that $\sigma$ is cohomologous to $\tau$, we will find a continuous function $b\colon G \to T$ satisfying $b(x) = 1$ for all $x \in \Go$, and
\[
\sigma(\alpha,\beta) = \tau(\alpha,\beta) \, b(\alpha) \, b(\beta) \, b(\alpha\beta)^{-1}
\]
for all $(\alpha,\beta) \in \Gc$. For each $\gamma \in G$, let $b(\gamma)$ be the unique element of $T$ such that $P(\gamma) = (\gamma, b(\gamma))$. Since $P(\Go) \subseteq \Go \times \{1\}$, we have $b(x) = 1$ for all $x \in \Go$. Since $b = \pi_2 \circ P$, where $\pi_2$ is the projection of $G \times_\tau T$ onto the second coordinate, $b$ is continuous. For all $(\alpha,\beta) \in \Gc$, we have
\begin{align*}
i\big(r(\alpha), \, \sigma(\alpha,\beta)\big) &= P(\alpha) P(\beta) P(\alpha\beta)^{-1} \\
&= (\alpha, b(\alpha)) \, (\beta, b(\beta)) \, (\alpha\beta, b(\alpha\beta))^{-1} \\
&= \big(\alpha\beta, \, \tau(\alpha,\beta) \, b(\alpha) \, b(\beta) \big) \big((\alpha\beta)^{-1}, \, \tau(\alpha\beta,(\alpha\beta)^{-1})^{-1} \, b(\alpha\beta)^{-1}\big) \\
&= \big(\alpha\beta (\alpha\beta)^{-1}, \, \tau(\alpha\beta, (\alpha\beta)^{-1}) \, \tau(\alpha,\beta) \, b(\alpha) \, b(\beta) \, \tau(\alpha\beta,(\alpha\beta)^{-1})^{-1} \, b(\alpha\beta)^{-1}\big) \\
&= \big(r(\alpha), \, \tau(\alpha,\beta) \, b(\alpha) \, b(\beta) \, b(\alpha\beta)^{-1}\big).
\end{align*}
Thus, noting that $i\colon \Go \times T \to G \times_\sigma T$ is the inclusion map, we deduce that
\[
\sigma(\alpha,\beta) = \tau(\alpha,\beta) \, b(\alpha) \, b(\beta) \, b(\alpha\beta)^{-1}
\]
for all $(\alpha,\beta) \in \Gc$, as required.
\end{proof}

We now show that cohomologous locally constant $2$-cocycles give rise to isomorphic twists.

\begin{lemma} \label{lemma: cohomologous implies isomorphism}
Let $G$ be a Hausdorff \'etale groupoid, and let $\sigma,\tau\colon \Gc \to T \le R^\times$ be continuous $2$-cocycles. If $\sigma$ is cohomologous to $\tau$, then the discrete twists $G \times_\sigma T$ and $G \times_\tau T$ are isomorphic.
\end{lemma}

\begin{proof}
Suppose that $\sigma$ is cohomologous to $\tau$. Then there is a continuous function $b\colon G \to T$ satisfying $b(x) = 1$ for all $x \in \Go$, and
\begin{equation} \label{eqn: sigma sim tau via b}
b(\alpha\beta) \, \sigma(\alpha,\beta) = \tau(\alpha,\beta) \, b(\alpha) \, b(\beta)
\end{equation}
for all $(\alpha,\beta) \in \Gc$. Define $\psi\colon G \times_\sigma T \to G \times_\tau T$ by $\psi(\alpha,z) \coloneqq (\alpha, b(\alpha) z)$. Then $\psi$ is bijective, with inverse given by $\psi^{-1}(\alpha,z) \coloneqq (\alpha, b(\alpha)^{-1} z)$. Since $\psi(\alpha,z) = (r(\alpha), b(\alpha)) (\alpha,z)$, $\psi$ is continuous, because it is the pointwise product of the continuous map $(r \times b) \circ \pi_1$ and the identity map, where $\pi_1$ is the projection of $G \times_\sigma T$ onto the first coordinate. A similar argument shows that $\psi^{-1}$ is continuous, and thus $\psi$ is a homeomorphism.

To see that $\psi$ is a groupoid homomorphism, fix $(\alpha,\beta) \in \Gc$ and $z, w \in T$. Using \cref{eqn: sigma sim tau via b} for the third equality, we obtain
\begin{align*}
\psi((\alpha,z)(\beta,w)) &= \psi(\alpha\beta, \, \sigma(\alpha,\beta) \, zw) \\
&= (\alpha\beta, \, b(\alpha\beta) \, \sigma(\alpha,\beta) \, zw) \\
&= (\alpha\beta, \, \tau(\alpha,\beta) \, b(\alpha) \, b(\beta) \, zw) \\
&= (\alpha, b(\alpha)z) \, (\beta, b(\beta)w) \\
&= \psi(\alpha,z) \, \psi(\beta,w),
\end{align*}
as required.

We have now shown that $G \times_\sigma T$ and $G \times_\tau T$ are isomorphic as groupoids. To see that they are isomorphic as discrete twists, let $i_\sigma\colon \Go \times T \to G \times_\sigma T$ and $i_\tau\colon \Go \times T \to G \times_\tau T$ be the inclusion maps, and let $q_\sigma \colon G \times_\sigma T \to G$ and $q_\tau \colon G \times_\tau T \to G$ be the projections onto the first coordinate. Since $b(x) = 1$ for all $x \in \Go$, we have
\[
\psi(i_\sigma(x,z)) = (x, b(x) z) = (x,z) = i_\tau(x,z),
\]
and
\[
q_\tau(\psi(\alpha,z)) = q_\tau(\alpha, b(\alpha) z) = \alpha = q_\sigma(\alpha),
\]
for all $x \in \Go$, $\alpha \in G$, and $z \in T$. Therefore, $\psi$ is an isomorphism of the twists $G \times_\sigma T$ and $G \times_\tau T$.
\end{proof}

Finally, we show that if $\sigma$ and $\tau$ are locally constant $2$-cocycles on $G$ giving rise to isomorphic discrete twists $G \times_\sigma T$ and $G \times_\tau T$, then $G \times_\tau T$ admits a continuous global section that induces $\sigma$.

\begin{lemma} \label{lemma: isomorphic twists each admit section inducing other cocycle}
Let $G$ be a Hausdorff \'etale groupoid, and let $\sigma,\tau\colon \Gc \to T \le R^\times$ be continuous $2$-cocycles. If $(G \times_\sigma T, i_\sigma, q_\sigma)$ and $(G \times_\tau T, i_\tau, q_\tau)$ are isomorphic as twists, then $\sigma$ is induced by a continuous global section $P\colon G \to G \times_\tau T$.
\end{lemma}

\begin{proof}
Suppose that $\psi\colon G \times_\sigma T \to G \times_\tau T$ is an isomorphism of twists. By \cref{lemma: trivial section induces sigma}, the map $S\colon \gamma \mapsto (\gamma,1)$ is a continuous global section from $G$ to $G \times_\sigma T$ that induces $\sigma$, in the sense that
\begin{equation} \label{eqn: S induces sigma}
S(\alpha) S(\beta) S(\alpha\beta)^{-1} = i_\sigma\big(r(\alpha), \, \sigma(\alpha,\beta)\big)
\end{equation}
for all $(\alpha,\beta) \in \Gc$.

Define $P\coloneqq \psi \circ S \colon G \to G \times_\tau T$. We claim that $P$ is a continuous global section. Since $S$ is a continuous global section and $\psi$ is a groupoid isomorphism, $P$ is continuous and $P(\Go) \subseteq \Go \times \{1\}$. Recall from \cref{eg: twisted groupoid from cocycle} that $q_\sigma\colon G \times_\sigma T \to G$ and $q_\tau\colon G \times_\tau T \to G$ are the projections onto the first coordinate. Since $\psi$ is an isomorphism of twists, we have
\[
q_\tau \circ P = q_\tau \circ (\psi \circ S) = (q_\tau \circ \psi) \circ S = q_\sigma \circ S = \id_G,
\]
and hence $P$ is a continuous global section.

We now show that $P$ induces $\sigma$. By \cref{prop: section induces cocycle}\cref{item: P induces sigma}, $P$ induces a continuous $2$-cocycle $\omega\colon \Gc \to T$ satisfying
\begin{equation} \label{eqn: P induces omega}
P(\alpha) P(\beta) P(\alpha\beta)^{-1} = i_\tau\big(r(\alpha), \, \omega(\alpha,\beta)\big)
\end{equation}
for all $(\alpha,\beta) \in \Gc$. Together, \cref{eqn: P induces omega,eqn: S induces sigma} imply that
\begin{align*}
i_\tau\big(r(\alpha), \, \omega(\alpha,\beta)\big) &= P(\alpha) P(\beta) P(\alpha\beta)^{-1} \\
&= \psi\big(S(\alpha) S(\beta) S(\alpha\beta)^{-1}\big) \\
&= \psi\big(i_\sigma\big(r(\alpha), \, \sigma(\alpha,\beta)\big)\big) \\
&= i_\tau\big(r(\alpha), \, \sigma(\alpha,\beta)\big),
\end{align*}
for all $(\alpha,\beta) \in \Gc$. Since $i_\sigma$ and $i_\tau$ are both injective, we deduce that $\sigma = \omega$, and hence $\sigma$ is induced by $P$.
\end{proof}

We now combine these three lemmas to prove our main theorem for this section.

\begin{proof}[Proof of \cref{thm: cohomology}]
\Cref{lemma: isomorphic twists each admit section inducing other cocycle} gives \cref{item: thm isomorphic} \!\!$\implies$\!\! \cref{item: thm induced}, \cref{lemma: section implies cohomologous} gives \cref{item: thm induced} \!\!$\implies$\!\! \cref{item: thm cohomologous}, and \cref{lemma: cohomologous implies isomorphism} gives \cref{item: thm cohomologous} \!\!$\implies$\!\! \cref{item: thm isomorphic}.
\end{proof}

We conclude this section with a corollary of \cref{thm: cohomology}.

\begin{cor}
Let $G$ be a Hausdorff \'etale groupoid, and let $\Sigma$ be a topologically trivial discrete twist by $T \le R^\times$ over $G$. Suppose that $\sigma,\tau\colon \Gc \to T$ are continuous $2$-cocycles that are induced by continuous global sections $P_\sigma,P_\tau\colon G \to \Sigma$, as in \cref{prop: section induces cocycle}\cref{item: P induces sigma}. Then $\sigma$ is cohomologous to $\tau$.
\end{cor}

\begin{proof}
By \cref{prop: section induces cocycle}\cref{item: P induces phi_P}, we have $G \times_\sigma T \cong \Sigma \cong G \times_\tau T$, and hence \cref{thm: cohomology} implies that $\sigma$ is cohomologous to $\tau$.
\end{proof}

\subsection{Twisted Steinberg algebras arising from discrete twists} \label{section: Steinberg algebras from twists}

In this section we give a construction of a twisted Steinberg algebra $A_R(G;\Sigma)$ coming from a topologically trivial discrete twist $\Sigma$ over an ample Hausdorff groupoid $G$. We prove that if two such twists are isomorphic, then they give rise to isomorphic twisted Steinberg algebras. We also prove that if $\Sigma \cong G \times_\sigma T$ for some continuous $2$-cocycle $\sigma\colon \Gc \to T \le R^\times$, then the twisted Steinberg algebras $A_R(G;\Sigma)$ and $A_R(G,\sigma^{-1})$ are $R$-algebraically isomorphic, where $\sigma^{-1}$ is the continuous $T$-valued $2$-cocycle $(\alpha,\beta) \mapsto \sigma(\alpha,\beta)^{-1}$.

\begin{definition} \label{def: A_R(G;Sigma)}
Let $G$ be an ample Hausdorff groupoid, and let $(\Sigma, i, q)$ be a topologically trivial discrete twist by $T \le R^\times$ over $G$. We say that $f \in C(\Sigma,R)$ is \hl{$T$-equivariant} if $f(z \cdot \varepsilon) = z \, f(\varepsilon)$ for all $z \in T$ and $\varepsilon \in \Sigma$, and we define
\[
A_R(G;\Sigma) \coloneqq \{ f \in C(\Sigma,R)\,:\, \text{$f$ is $T$-equivariant and $\overline{q(\supp(f))}$ is compact} \}.
\]
\end{definition}

We first show that $A_R(G;\Sigma)$ is an $R$-module under the pointwise operations inherited from $C(\Sigma,R)$.

\begin{lemma} \label{lemma: A_R(G;Sigma) is an R-module}
Let $G$ be an ample Hausdorff groupoid, and let $(\Sigma, i, q)$ be a topologically trivial discrete twist by $T \le R^\times$ over $G$. Then $A_R(G;\Sigma)$ is an $R$-submodule of $C(\Sigma,R)$.
\end{lemma}

\begin{proof}
Fix $f, g \in A_R(G;\Sigma)$ and $\lambda \in R$. Then $\lambda f + g$ is continuous and $T$-equivariant. Since $q\big(\!\supp(\lambda f + g)\big)$ is contained in the compact set $\overline{q(\supp(f))} \medcup \overline{q(\supp(g))}$, we deduce that $q\big(\!\supp(\lambda f + g)\big)$ has compact closure. Hence $\lambda f + g \in A_R(G;\Sigma)$.
\end{proof}

Since we are assuming that the twist $\Sigma$ is topologically trivial, it necessarily admits a continuous global section $P\colon G \to \Sigma$. We now show that \cref{def: A_R(G;Sigma)} can be rephrased in terms of any such $P$.

\begin{lemma} \label{lemma: A_R(G;Sigma) in terms of P}
Let $G$ be an ample Hausdorff groupoid, and let $(\Sigma, i, q)$ be a topologically trivial discrete twist by $T \le R^\times$ over $G$. Let $P\colon G \to \Sigma$ be any continuous global section. Then
\[
A_R(G;\Sigma) = \{ f \in C(\Sigma,R)\,:\, \text{$f$ is $T$-equivariant and $f \circ P \in C_c(G,R)$} \}.
\]
\end{lemma}

\begin{proof}
Fix $f \in C(\Sigma,R)$. Then $f \circ P$ is continuous. It suffices to show that $q(\supp(f)) = \supp(f \circ P)$, because then $\overline{q(\supp(f))}$ is compact if and only if $f \circ P \in C_c(G,R)$. By \cref{prop: section induces cocycle}\cref{item: P induces phi_P}, we know that $\Sigma = \{ z \cdot P(\alpha) : (\alpha,z) \in G \times T \}$. Therefore, we have
\begin{align*}
q(\supp(f)) &= \{ q(\varepsilon) \,:\, \varepsilon \in \Sigma, \, f(\varepsilon) \ne 0 \} \\
&= \{ q(z \cdot P(\alpha)) \,:\, (\alpha,z) \in G \times T, \, f(z \cdot P(\alpha)) \ne 0 \} \\
&= \{ \alpha \,:\, (\alpha,z) \in G \times T, \ z \, f(P(\alpha)) \ne 0 \} \\
&= \{ \alpha \in G \,:\, (f \circ P)(\alpha) \ne 0 \} \\
&= \supp(f \circ P),
\end{align*}
as required.
\end{proof}

\begin{remarks}[On the relationship with the classical setting] \label{remarks: why discretise} \leavevmode
\begin{enumerate}[label=(\arabic*), ref={\cref{remarks: why discretise}(\arabic*)}]
\item It is crucial here that we are dealing with discrete twists. Suppose that $\sigma$ is a $\T$-valued $2$-cocycle on an ample Hausdorff groupoid $G$ that is continuous with respect to the standard topology on $\T$, and consider the classical twist $G \times_\sigma \T$ over $G$. Suppose that $f \in C(G \times_\sigma \T)$ is a $\T$-equivariant function that is locally constant. Then, for any $\alpha \in G$, there is an open subset $V$ of $G$ containing $\alpha$ and an open subset $W$ of $\T$ containing $1$ such that $f$ is constant on $V \times W$. Since $W$ is open in the standard topology on $\T$, we have $W \ne \{1\}$. For each $z \in W {\setminus} \{1\}$, we have
\[
f(\alpha,1) = f(\alpha,z) = f(z \cdot (\alpha,1)) = z \, f(\alpha,1),
\]
and hence $f\restr{G \times \{1\}} \equiv 0$. But this implies that $f(\beta,w) = 0$ for all $(\beta,w) \in G \times_\sigma \T$, because $f$ is $\T$-equivariant. In other words, if singleton sets are not open in $\T$, then the only locally constant $\T$-equivariant function on $G \times_\sigma \T$ is the zero function.

\item \label{item: no C_c equivariant functions} It is also crucial that \Cref{def: A_R(G;Sigma)} differs from the C*-algebraic analogue defined in \cite[Definition~11.1.7 and Theorem~11.1.11]{Sims2020}, which is a C*-completion of the subalgebra of continuous \emph{compactly supported} $\T$-equivariant functions on a (classical) twist over $G$. To see why the compact-support condition would not be appropriate in the discrete setting, suppose that $G$ is an ample Hausdorff groupoid, and that $\sigma\colon \Gc \to \T_d$ is a continuous $2$-cocycle. Since $\T_d$ has the discrete topology, nonzero functions in $A_{\C_d}(G; \, G \times_\sigma \T_d)$ are not compactly supported. To see this, fix $f \in A_{\C_d}(G; \, G \times_\sigma \T_d)$ such that $f(\alpha,w) \ne 0$ for some $(\alpha,w) \in G\times_\sigma \T_d$. Then, for all $z \in \T_d$, we have
\[
f(\alpha,z) = f(\alpha, z \, \overline{w} \, w) = f\big((z \, \overline{w}) \cdot (\alpha,w)\big) = z \, \overline{w} \, f(\alpha,w) \ne 0.
\]
Thus $\{\alpha\} \times \T_d$ is a closed subset of $\supp(f)$ which is not compact (because $\T_d$ is not compact), and hence $f$ is not compactly supported.
\end{enumerate}
\end{remarks}

\begin{prop} \label{prop: Steinberg algebra of a twist}
Let $G$ be an ample Hausdorff groupoid, and let $(\Sigma, i, q)$ be a topologically trivial discrete twist by $T \le R^\times$ over $G$. Let $P\colon G \to \Sigma$ be any continuous global section. There is a multiplication (called \hl{convolution}) on the $R$-module $A_R(G;\Sigma)$, given by
\begin{equation} \label{eqn: multiplication *_Sigma}
(f*_\Sigma g)(\varepsilon) \coloneqq \sum_{\gamma \in G^{s(q(\varepsilon))}} f(\varepsilon P(\gamma)) \, g(P(\gamma)^{-1}),
\end{equation}
under which $A_R(G;\Sigma)$ is an $R$-algebra. We call $A_R(G;\Sigma)$ the \hl{twisted Steinberg algebra} associated to the pair $(G,\Sigma)$. If $R$ has a $T$-inverse involution $r \mapsto \overline{r}$, then there is also an involution on $A_R(G;\Sigma)$, given by
\[
f^*(\varepsilon) \coloneqq \overline{f(\varepsilon^{-1})},
\]
under which $A_R(G;\Sigma)$ is a $*$-algebra over $R$.
\end{prop}

\begin{proof}
By \cref{lemma: A_R(G;Sigma) is an R-module}, $A_R(G;\Sigma)$ is an $R$-module. We first show that the multiplication formula given in \cref{eqn: multiplication *_Sigma} is well-defined. To see this, fix $f, g \in A_R(G;\Sigma)$, and suppose that $P, S\colon G \to \Sigma$ are continuous global sections. For each $\gamma \in G$, we have $q(P(\gamma)) = \gamma = q(S(\gamma))$, and hence by \cref{lemma: q and action}, there exists a unique $z_\gamma \in T$ such that $P(\gamma) = z_\gamma \cdot S(\gamma)$. Fix $\varepsilon \in \Sigma$ and $\gamma \in G^{s(q(\varepsilon))}$. Since $f$ and $g$ are $T$-equivariant, we have
\begin{align*}
f(\varepsilon P(\gamma)) \, g(P(\gamma)^{-1}) &= f\big(z_\gamma \cdot (\varepsilon S(\gamma))\big) \, g\big(z_\gamma^{-1} \cdot S(\gamma)^{-1}\big) \\
&= z_\gamma \, f(\varepsilon S(\gamma)) \, z_\gamma^{-1} \, g(S(\gamma)^{-1}) \\
&= f(\varepsilon S(\gamma)) \, g(S(\gamma)^{-1}),
\end{align*}
and so the sum defining $f *_\Sigma g$ is independent of the choice of continuous global section. To see that the sum in \cref{eqn: multiplication *_Sigma} is finite, observe that since $f$ and $g$ are $T$-equivariant, \cref{lemma: q and action} implies that $\varepsilon P(\gamma) \in \supp(f)$ if and only if $q(\varepsilon) \gamma \in q(\supp(f))$, and $P(\gamma)^{-1} \in \supp(g)$ if and only if $\gamma^{-1} \in q(\supp(g))$. Since $\overline{q(\supp(f))}$ and $\overline{q(\supp(g))}$ are compact and $G^{s(q(\varepsilon))}$ is discrete, it follows that the set
\[
\left\{ \gamma \in G^{s(q(\varepsilon))} : f(\varepsilon P(\gamma)) \, g(P(\gamma)^{-1}) \ne 0 \right\} \,\subseteq\, G^{s(q(\varepsilon))} \,\medcap\, q(\varepsilon)^{-1} q(\supp(f)) \,\medcap\, q(\supp(g))^{-1}
\]
is finite, and hence $f *_\Sigma g$ is well-defined.

To see that $A_R(G;\Sigma)$ is an $R$-algebra, we will just show that it is closed under the multiplication, as it is routine to check that the multiplication satisfies all of the other necessary properties. Recall that by \cref{prop: section induces cocycle}, $P$ induces a continuous $2$-cocycle $\sigma\colon \Gc \to T \le R^\times$ such that the map $\phi_P\colon G \times_\sigma T \to \Sigma$ given by $\phi_P(\alpha,z) \coloneqq z \cdot P(\alpha)$ is an isomorphism of twists. Fix $f, g \in A_R(G;\Sigma)$, and define $f_P \coloneqq f \circ P$ and $g_P \coloneqq g \circ P$. By \cref{lemma: A_R(G;Sigma) in terms of P}, $f_P$ and $g_P$ are elements of $C_c(G,R)$, which is equal (as an $R$-module) to $A_R(G,\sigma^{-1})$, by \cref{lemma: functions in twisted Steinberg algebras}\cref{item: functions are C_c and LC}. We will express the product $f *_\Sigma g$ in terms of $f_P *_{\sigma^{-1}} g_P$, which we know is an element of $A_R(G,\sigma^{-1})$, by \cref{prop: twisted Steinberg algebra}. Fix $(\alpha,z) \in G \times_\sigma T$. Using $T$-equivariance for the second and fourth equalities and \cref{prop: section induces cocycle}\cref{item: P properties} for the third equality below, we obtain
\begin{align*}
(f *_\Sigma g)(z \cdot P(\alpha)) \,&=\, \sum_{\beta \in G^{s(q(z \cdot P(\alpha)))}} f\big((z \cdot P(\alpha)) \, P(\beta)\big) \, g(P(\beta)^{-1}) \\
&=\, \sum_{\beta \in G^{s(\alpha)}} z \, f(P(\alpha)P(\beta)) \, g(P(\beta)^{-1}) \\
&=\, z \, \sum_{\beta \in G^{s(\alpha)}} f\big(\sigma(\alpha,\beta) \cdot P(\alpha\beta)\big) \, g\big(\sigma(\beta,\beta^{-1})^{-1} \cdot P(\beta^{-1})\big) \\
&=\, z \, \sum_{\beta \in G^{s(\alpha)}} \sigma(\alpha,\beta) \, \sigma(\beta,\beta^{-1})^{-1} \, f_P(\alpha\beta) \, g_P(\beta^{-1}). \numberthis \label{eqn: *_Sigma convolution LHS}
\end{align*}
We also have
\begin{equation} \label{eqn: * convolution RHS}
(f_P *_{\sigma^{-1}} g_P)(\alpha) \,=\, \sum_{\beta \in G^{s(\alpha)}} \sigma^{-1}(\alpha\beta,\beta^{-1}) \, f_P(\alpha\beta) \, g_P(\beta^{-1}).
\end{equation}
Since $\sigma$ is normalised and satisfies the $2$-cocycle identity, we have
\[
\sigma(\alpha,\beta) \, \sigma(\alpha\beta,\beta^{-1}) = \sigma(\alpha,\beta\beta^{-1}) \, \sigma(\beta,\beta^{-1}) = \sigma(\beta,\beta^{-1}),
\]
and hence
\begin{equation} \label{eqn: cocycle terms in *_Sigma}
\sigma(\alpha,\beta) \, \sigma(\beta,\beta^{-1})^{-1} = \sigma(\alpha\beta,\beta^{-1})^{-1} = \sigma^{-1}(\alpha\beta,\beta^{-1}),
\end{equation}
for each $\beta \in G^{s(\alpha)}$. Together, \cref{eqn: *_Sigma convolution LHS,eqn: * convolution RHS,eqn: cocycle terms in *_Sigma} imply that
\begin{equation} \label{eqn: relating *_Sigma to *}
(f *_\Sigma g)(\phi_P(\alpha,z)) = (f *_\Sigma g)(z \cdot P(\alpha)) = z \, (f_P *_{\sigma^{-1}} g_P)(\alpha).
\end{equation}
Define $\psi_P^{f,g}\colon G \times_\sigma T \to R$ by $\psi_P^{f,g}(\alpha,z) \coloneqq z \, (f_P *_{\sigma^{-1}} g_P)(\alpha)$. Since $f_P, g_P \in A_R(G,\sigma^{-1})$, we have $f_P *_{\sigma^{-1}} g_P \in A_R(G,\sigma^{-1}) \subseteq C(G,R)$. Thus $\psi_P^{f,g}$ is continuous. Since $\phi_P$ is a homeomorphism and $f *_\Sigma g = \psi_P^{f,g} \circ \phi_P^{-1}$, we deduce that $f *_\Sigma g \in C(\Sigma,R)$. Taking $z = 1$ in \cref{eqn: relating *_Sigma to *} shows that $(f *_\Sigma g) \circ P = f_P *_{\sigma^{-1}} g_P \in C_c(G,R)$, and \cref{lemma: A_R(G;Sigma) in terms of P} implies that this is equivalent to showing that $\overline{q(\supp(f *_\Sigma g))}$ is compact. Finally, to see that $f *_\Sigma g$ is $T$-equivariant, fix $z \in T$ and $\varepsilon \in \Sigma$. Then $\varepsilon = w \cdot P(\beta)$ for a unique pair $(\beta,w) \in G \times_\sigma T$. Thus, \cref{eqn: relating *_Sigma to *} implies that $(f *_\Sigma g)(\varepsilon) = w \, (f_P *_{\sigma^{-1}} g_P)(\beta)$, and hence
\[
(f *_\Sigma g)(z \cdot \varepsilon) = (f *_\Sigma g)\big((zw) \cdot P(\beta)\big) = z \, w \, (f_P *_{\sigma^{-1}} g_P)(\beta) = z \, (f *_\Sigma g)(\varepsilon).
\]
Therefore, $f *_\Sigma g \in A_R(G;\Sigma)$, and so $A_R(G;\Sigma)$ is an $R$-algebra.

Suppose now that $R$ has a $T$-inverse involution $r \mapsto \overline{r}$. We show that $f^* \in A_R(G;\Sigma)$. Since $f$ is continuous, $f^*$ is a composition of continuous maps, and so $f^* \in C(\Sigma,R)$. For all $z \in T$ and $\varepsilon \in \Sigma$, we have
\[
f^*(z \cdot \varepsilon) = \overline{f((z \cdot \varepsilon)^{-1})} = \overline{f\big((z^{-1}) \cdot (\varepsilon^{-1})\big)} = \overline{z^{-1} f(\varepsilon^{-1})} = z \, f^*(\varepsilon),
\]
and so $f^*$ is $T$-equivariant. Since $\supp(f^*) = (\supp(f))^{-1}$ and $q$ is a continuous homomorphism, we have $q(\supp(f^*)) \subseteq \big(\overline{q(\supp(f))}\big)^{-1}$, and hence $\overline{q(\supp(f^*))}$ is compact because it is a closed subset of a compact set. Thus $f^* \in A_R(G;\Sigma)$. Routine calculations show that the map $f \mapsto f^*$ satisfies all of the properties of an involution on $A_R(G;\Sigma)$, since $r \mapsto \overline{r}$ is an involution on $R$. Therefore, $A_R(G;\Sigma)$ is a $*$-algebra over $R$.
\end{proof}

We now show that isomorphic twists give rise to isomorphic twisted Steinberg algebras.

\begin{prop} \label{prop: Steinberg algebras of isomorphc twists}
Let $G$ be an ample Hausdorff groupoid. Suppose that $(\Sigma_1, i_1, q_1)$ and $(\Sigma_2, i_2, q_2)$ are topologically trivial discrete twists by $T \le R^\times$ over $G$. If $\psi\colon \Sigma_1 \to \Sigma_2$ is an isomorphism of twists, then the map $\Phi\colon f \mapsto f \circ \psi$ is an isomorphism from $A_R(G;\Sigma_2)$ to $A_R(G;\Sigma_1)$. If $R$ has a $T$-inverse involution, then $\Phi$ is a $*$-isomorphism.
\end{prop}

\begin{proof}
We first show that $f \circ \psi \in A_R(G;\Sigma_1)$ for each $f \in A_R(G;\Sigma_2)$. Let $P_1\colon G \to \Sigma_1$ be a continuous global section, and define $P_2 \coloneqq \psi \circ P_1\colon G \to \Sigma_2$. Then $P_2$ is continuous, $P_2(\Go) \subseteq \psi\big(\Sigma_1^{(0)}\big) = \Sigma_2^{(0)}$, and since $q_2 \circ \psi = q_1$,
\[
q_2 \circ P_2 = q_2 \circ (\psi \circ P_1) = (q_2 \circ \psi) \circ P_1 = q_1 \circ P_1 = \id_G.
\]
Hence $P_2$ is a continuous global section. Fix $f \in A_R(G;\Sigma_2) \subseteq C(\Sigma_2,R)$. Since $\psi$ is continuous, $f \circ \psi \in C(\Sigma_1,R)$. By \cref{lemma: twist isomorphisms respect T-action}, $\psi$ respects the action of $T$, and hence the $T$-equivariance of $f$ implies that $f \circ \psi$ is $T$-equivariant. Moreover, \cref{lemma: A_R(G;Sigma) in terms of P} implies that $f \circ \psi \circ P_1 = f \circ P_2 \in C_c(G,R)$, and thus $f \circ \psi \in A_R(G;\Sigma_1)$.

Therefore, there is a map $\Phi\colon A_R(G;\Sigma_2) \to A_R(G;\Sigma_1)$ given by $\Phi(f) \coloneqq f \circ \psi$. Routine calculations show that $\Phi$ is a homomorphism, and that if $R$ has a $T$-inverse involution, then $\Phi$ is a $*$-homomorphism. Furthermore, $\Phi$ is bijective with inverse given by $\Phi^{-1}(g) = g \circ \psi^{-1}$, and hence $\Phi$ is an isomorphism (or a $*$-isomorphism).
\end{proof}

By \cref{prop: section induces cocycle}, we know that for every topologically trivial discrete twist $\Sigma$ over an ample Hausdorff groupoid $G$, there is a continuous $2$-cocycle $\sigma\colon \Gc \to T \le R^\times$ such that $\Sigma \cong G \times_\sigma T$. Hence $A_R(G;\Sigma)$ is isomorphic to $A_R(G; \, G \times_\sigma T)$, by \cref{prop: Steinberg algebras of isomorphc twists}. We now prove that $A_R(G;\Sigma)$ is also isomorphic to $A_R(G,\sigma^{-1})$.

\begin{thm} \label{thm: same twisted Steinberg algebras}
Let $G$ be an ample Hausdorff groupoid, and let $\Sigma$ be a topologically trivial discrete twist by $T \le R^\times$ over $G$. Let $P\colon G \to \Sigma$ be a continuous global section, and let $\sigma\colon \Gc \to T$ be the continuous $2$-cocycle induced by $P$, as in \cref{prop: section induces cocycle}\cref{item: P induces sigma}. The map $\Psi\colon f \mapsto f \circ P$ is an isomorphism from $A_R(G;\Sigma)$ to $A_R(G,\sigma^{-1})$. If $R$ has a $T$-inverse involution, then $\Psi$ is a $*$-isomorphism.
\end{thm}

\begin{remark}
In the C*-setting, some authors (for example, \cite{BaH2014}) define the twisted groupoid C*-algebra $C^*(G;\Sigma)$ to be a C*-completion of the set of \emph{$\T$-contravariant} functions in $C_c(\Sigma)$, rather than $\T$-equivariant functions; that is,
\[
\{ f \in C_c(\Sigma) : f(z \cdot \varepsilon) = \overline{z} \, f(\varepsilon) \text{ for all } z \in \T, \, \varepsilon \in \Sigma \},
\]
rather than
\[
\{ f \in C_c(\Sigma) : f(z \cdot \varepsilon) = z \, f(\varepsilon) \text{ for all } z \in \T, \, \varepsilon \in \Sigma \}.
\]
As a consequence of this definition, the C*-analogue of \cref{thm: same twisted Steinberg algebras} gives an isomorphism between $C^*(G;\Sigma)$ and $C^*(G,\sigma)$, rather than $C^*(G;\Sigma)$ and $C^*(G,\sigma^{-1})$. Similarly, an alternate definition of $A_R(G;\Sigma)$ consisting of $T$-contravariant functions would result in $A_R(G;\Sigma)$ being isomorphic to $A_R(G,\sigma)$.
\end{remark}

\begin{proof}[Proof of \cref{thm: same twisted Steinberg algebras}]
By \cref{lemma: functions in twisted Steinberg algebras}\cref{item: functions are C_c and LC}, $A_R(G,\sigma^{-1})$ and $C_c(G,R)$ agree as $R$-modules, and hence \cref{lemma: A_R(G;Sigma) in terms of P} implies that
\begin{equation} \label{eqn: f circ P expression of A_R(G;Sigma)}
A_R(G;\Sigma) = \{ f \in C(\Sigma,R)\,:\, f \text{ is $T$-equivariant and } f \circ P \in A_R(G,\sigma^{-1}) \}.
\end{equation}
Therefore, there is a map $\Psi\colon A_R(G;\Sigma) \to A_R(G,\sigma^{-1})$ given by $\Psi(f) \coloneqq f \circ P$.

To see that $\Psi$ is injective, suppose that $\Psi(f) = \Psi(g)$ for some $f, g \in A_R(G;\Sigma)$. Fix $(\alpha,z) \in G \times_\sigma T$. Since $f$ and $g$ are $T$-equivariant, we have
\begin{equation} \label{eqn: injectivity of psi}
f(z \cdot P(\alpha)) = z \, f(P(\alpha)) = z \, \Psi(f)(\alpha) = z \, \Psi(g)(\alpha) = z \, g(P(\alpha)) = g(z \cdot P(\alpha)).
\end{equation}
By \cref{prop: section induces cocycle}\cref{item: P induces phi_P}, we have $\Sigma = \{ z \cdot P(\alpha) : (\alpha,z) \in G \times_\sigma T \}$, and so \cref{eqn: injectivity of psi} implies that $f = g$, and hence $\Psi$ is injective.

To see that $\Psi$ is surjective, fix $h \in A_R(G,\sigma^{-1})$, and recall from \cref{prop: section induces cocycle}\cref{item: P induces phi_P} that the map $\phi_P\colon G \times_\sigma T \to \Sigma$ given by $\phi_P(\alpha,z) \coloneqq z \cdot P(\alpha)$ is an isomorphism of twists. Define $f\colon \Sigma \to R$ by $f(z \cdot P(\alpha)) \coloneqq z \, h(\alpha)$, and $\widetilde{f}\colon G \times_\sigma T \to R$ by $\widetilde{f}(\alpha,z) \coloneqq z \, h(\alpha)$. Since $h \in C(G,R)$, we have $\widetilde{f} \in C(G \times_\sigma T, \, R)$, and hence $f = \widetilde{f} \circ \phi_P^{-1} \in C(\Sigma,R)$ because $\phi_P^{-1}$ is continuous. For all $\alpha \in G$ and $z, w \in T$, we have
\[
f\big(z \cdot (w \cdot P(\alpha))\big) = f\big((zw) \cdot P(\alpha)\big) = z \, w \, h(\alpha) = z \, f(w \cdot P(\alpha)),
\]
and so $f$ is $T$-equivariant. We also have $f \circ P = h \in A_R(G,\sigma^{-1})$, and thus \cref{eqn: f circ P expression of A_R(G;Sigma)} implies that $f \in A_R(G;\Sigma)$. Since $\Psi(f) = f \circ P = h$, $\Psi$ is surjective.

It is clear that $\Psi$ is $R$-linear. We claim that $\Psi$ is an $R$-algebra isomorphism. Fix $f, g \in A_R(G;\Sigma)$. In the notation introduced in the proof of \cref{prop: Steinberg algebra of a twist}, we have $\Psi(f) = f_P$ and $\Psi(g) = g_P$, and hence \cref{eqn: relating *_Sigma to *} implies that for all $\alpha \in G$, we have
\[
\Psi(f *_\Sigma g)(\alpha) = (f *_\Sigma g)(P(\alpha)) = \big(\Psi(f) *_{\sigma^{-1}} \Psi(g)\big)(\alpha).
\]
So $\Psi(f *_\Sigma g) = \Psi(f) *_{\sigma^{-1}} \Psi(g)$, and thus $\Psi$ is an isomorphism.

Suppose now that $R$ has a $T$-inverse involution $r \mapsto \overline{r}$. To see that $\Psi$ is a $*$-isomorphism, we must show that $\Psi(f^*) = \Psi(f)^*$. Fix $\alpha \in G$. By \cref{prop: section induces cocycle}\cref{item: P properties}, we have
\[
P(\alpha)^{-1} = \sigma^{-1}(\alpha,\alpha^{-1}) \cdot P(\alpha^{-1}),
\]
and hence
\begin{equation} \label{eqn: psi(f*)}
\Psi(f^*)(\alpha) = f^*(P(\alpha)) = \overline{f\big(P(\alpha)^{-1}\big)} = \overline{f\big(\sigma^{-1}(\alpha,\alpha^{-1}) \cdot P(\alpha^{-1})\big)}.
\end{equation}
We also have
\begin{align*}
\Psi(f)^*(\alpha) &= \big(\sigma^{-1}(\alpha,\alpha^{-1})\big)^{-1} \, \overline{\Psi(f)(\alpha^{-1})} \\
&= \overline{\sigma^{-1}(\alpha,\alpha^{-1}) \, f(P(\alpha^{-1}))} \\
&= \overline{f\big(\sigma^{-1}(\alpha,\alpha^{-1}) \cdot P(\alpha^{-1})\big)}. \numberthis \label{eqn: psi(f)*}
\end{align*}
Together, \cref{eqn: psi(f*),eqn: psi(f)*} imply that $\Psi(f^*) = \Psi(f)^*$.
\end{proof}

\begin{cor}
Let $G$ be an ample Hausdorff groupoid, and let $\sigma\colon \Gc \to T \le R^\times$ be a continuous $2$-cocycle. There is an isomorphism $\Psi\colon A_R(G; \, G \times_\sigma T) \to A_R(G,\sigma^{-1})$ such that $\Psi(f)(\gamma) = f(\gamma,1)$ for all $f \in A_R(G; \, G \times_\sigma T)$ and $\gamma \in G$. If $R$ has a $T$-inverse involution, then $\Psi$ is a $*$-isomorphism.
\end{cor}

\begin{proof}
By \cref{lemma: trivial section induces sigma}, the map $S\colon \gamma \mapsto (\gamma,1)$ is a continuous global section from $G$ to $G \times_\sigma T$ that induces $\sigma$, and so the result follows from \cref{thm: same twisted Steinberg algebras}.
\end{proof}

\begin{remark}
If $G$ is an ample Hausdorff groupoid, then $G \times_\sigma \T_d$ is also an ample Hausdorff groupoid for any continuous $2$-cocycle $\sigma\colon \Gc \to \T_d$, and hence there is an associated (untwisted) complex Steinberg algebra $A(G \times_\sigma \T_d)$. As a vector space, $A(G \times_\sigma \T_d)$ is equal to $C_c(G \times_\sigma \T_d, \, \C_d)$ and is dense in $C_r^*(G \times_\sigma \T_d)$, by \cite[Proposition~4.2]{CFST2014} and \cite[Proposition~5.7]{Steinberg2010}. Moreover, by \cref{thm: same twisted Steinberg algebras}, we have $A(G; \, G \times_\sigma \T_d) \cong A(G,\sigma^{-1})$, and we know from \cref{prop: twisted Steinberg algebra} that $A(G,\sigma^{-1})$ is dense in $C_r^*(G,\sigma^{-1})$. We saw in \cref{item: no C_c equivariant functions} that the only compactly supported function in $A(G; \, G \times_\sigma \T_d) \subseteq C(G \times_\sigma \T_d, \, \C_d)$ is the zero function, and hence
\[
A(G; \, G \times_\sigma \T_d) \,\medcap\, A(G \times_\sigma \T_d) \,=\, \{0\}.
\]
However, this does not preclude $C_r^*(G,\sigma^{-1})$ from embedding into $C_r^*(G \times_\sigma \T_d)$. It would be interesting to know how these two C*-algebras are related.
\end{remark}

\section{Examples of twisted Steinberg algebras} \label{section: examples}

In this section we discuss two important classes of examples of twisted Steinberg algebras: twisted group algebras and twisted Kumjian--Pask algebras.

\subsection{Twisted discrete group algebras}

Suppose that $R$ is a discrete commutative unital ring and that $G$ is a topological group (that is, $G$ is a group endowed with a topology with respect to which multiplication and inversion are continuous.) Then $G$ is an ample groupoid if and only if $G$ has the discrete topology, in which case, any $R^\times$-valued $2$-cocycle on $G$ is locally constant. One defines a twist over a discrete group $G$ via a split extension by an abelian group $A$, as in \cite[Chapter~IV.3]{Brown1982}. When $A = R^\times$, the twist gives rise to an $R^\times$-valued $2$-cocycle on $G$, with which one can define a twisted group $R$-algebra. The twisted convolution and involution defined in \cref{prop: twisted Steinberg algebra} generalise those of classical twisted group algebras over $R^\times$, and hence our twisted Steinberg algebras generalise these twisted (discrete) group algebras. Interesting open questions about this class of algebras still exist, even for finite groups. (See, for example, \cite{MS2018}.) Moreover, twisted group C*-algebras (as studied in \cite{PR1992}) have featured prominently in the study of C*-algebras associated with groups and group actions; in particular, they have proved essential in establishing superrigidity results for certain nilpotent groups (see \cite{ER2018}).

\subsection{Twisted Kumjian--Pask algebras}

For each finitely aligned higher-rank graph (or $k$-graph) $\Lambda$, there is both a C*-algebra $C^*(\Lambda)$ called the \hl{Cuntz--Krieger algebra} (see \cite{RSY2004}) and a dense subalgebra $\KP(\Lambda)$ called the \hl{Kumjian--Pask algebra} (see \cite{APCaHR2013, CP2017}) encoding the structure of the graph. Letting $G_\Lambda$ denote the boundary-path groupoid defined in \cite{KP2000, FMY2005, Yeend2007}, we have
\[
C^*(\Lambda) \cong C^*(G_\Lambda) \quad \text{and} \quad \KP(\Lambda) \cong A(G_\Lambda).
\]
Twisted higher-rank graph C*-algebras were introduced and studied in a series of papers by Kumjian, Pask, and Sims \cite{KPS2012, KPS2013, KPS2015, KPS2016}, and they provide a class of (somewhat) tractable examples that can be used to demonstrate more general C*-algebraic phenomena. (See also \cite{AB2018, Gillaspy2015, SWW2014}.) We introduce twisted Kumjian--Pask algebras for row-finite higher-rank graphs with no sources using a twisted Steinberg algebra approach.

Let $\Lambda$ be a row-finite higher-rank graph with no sources, and let $c$ be a continuous $\T$-valued $2$-cocycle on $\Lambda$, as defined in \cite[Definition~3.5]{KPS2015}. Then $C^*(\Lambda,c)$ is the C*-algebra generated by a universal Cuntz--Krieger $(\Lambda,c)$-family, as defined in \cite[Definition~5.2]{KPS2015}. In \cite[Theorem~6.3(iii)]{KPS2015}, the authors describe how $\Lambda$ and $c$ give rise to a $2$-cocycle $\sigma_c\colon G_\Lambda^{(2)} \to \T$ such that
\[
C^*(\Lambda,c) \cong C^*(G_\Lambda,\sigma_c).
\]
By the last two sentences of the proof of \cite[Lemma~6.3]{KPS2015}, the $2$-cocycle $\sigma_c$ is normalised and locally constant. We define
\[
\KP(\Lambda,c) \coloneqq A(G_\Lambda,\sigma_c),
\]
and call this the \hl{(complex) twisted Kumjian--Pask algebra} associated to the pair $(\Lambda,c)$. By \cref{prop: twisted Steinberg algebra}, $\KP(\Lambda,c)$ is dense in $C^*(\Lambda,c)$.

In \cite[Definition~5.2]{KPS2015}, Kumjian, Pask, and Sims construct $C^*(\Lambda,c)$ using a generators-and-relations model involving the same generating partial isometries $\{ t_\lambda : \lambda \in \Lambda \}$ as $C^*(\Lambda)$, but with the relation $t_\mu t_\nu = t_{\mu\nu}$ replaced by $t_\mu t_\nu = c(\mu,\nu) \, t_{\mu\nu}$. We expect that there is a similar construction of $\KP(\Lambda,c)$ using these generators and relations, but we do not pursue this here.

\section{A Cuntz--Krieger uniqueness theorem and simplicity of twisted Steinberg algebras of effective groupoids} \label{section: uniqueness and simplicity}

In this section we extend the Cuntz--Krieger uniqueness theorem and a part of the simplicity characterisation for Steinberg algebras from \cite{BCFS2014} to the twisted Steinberg algebra setting. Throughout this section, we will assume that $G$ is an effective ample Hausdorff groupoid, and that $R = \F_d$ is a field endowed with the discrete topology.

\begin{thm}[Cuntz--Krieger uniqueness theorem] \label{thm: CK uniqueness}
Let $\F_d$ be a discrete field, let $G$ be an effective ample Hausdorff groupoid, and let $\sigma\colon \Gc \to \F_d^\times$ be a continuous $2$-cocycle. Suppose that $Q$ is a ring and that $\pi\colon A_{\F_d}(G,\sigma) \to Q$ is a ring homomorphism. Then $\pi$ is injective if and only if $\pi(1_V) \ne 0$ for every nonempty compact open subset $V$ of $\Go$.
\end{thm}

\begin{proof}
It is clear that if $\pi$ is injective, then $\pi(1_V) \ne 0$ for every nonempty compact open subset $V$ of $\Go$. Suppose that $\pi$ is not injective. Then there exists $f \in A_{\F_d}(G,\sigma)$ such that $f \ne 0$ and $\pi(f) = 0$. We aim to find a nonempty compact open subset $V$ of $\Go$ such that $\pi(1_V) = 0$. Since $\sigma$ is locally constant, we can use \cref{lemma: functions in twisted Steinberg algebras}\cref{item: bisection sum} to write $f = \sum_{D \in F} a_D 1_D$, where $F$ is a finite collection of disjoint nonempty compact open bisections of $G$ such that $\sigma(\alpha^{-1},\alpha)$ is constant for all $\alpha \in D$, and $a_D \in \F_d {\setminus} \{0\}$, for each $D \in F$. Let $g \coloneqq 1_{D_0^{-1}} f$ for some $D_0 \in F$. Then $g \in \ker(\pi)$, because $\pi$ is a homomorphism. Fix $\alpha \in D_0$, and define $c_{D_0} \coloneqq \sigma(\alpha^{-1},\alpha) \, a_{D_0} \ne 0$. Then
\begin{equation} \label{eqn: g(s(alpha)) is nonzero}
g(s(\alpha)) = g(\alpha^{-1}\alpha) = \sigma(\alpha^{-1},\alpha) \, 1_{D_0^{-1}}(\alpha^{-1}) \, f(\alpha) = \sigma(\alpha^{-1},\alpha) \, a_{D_0} = c_{D_0} \ne 0.
\end{equation}
Define $g_0\colon G \to \F_d$ by
\[
g_0(\gamma) \coloneqq \begin{cases}
g(\gamma) & \ \text{if } \gamma \in \Go \\
0 & \ \text{if } \gamma \in G {\setminus} \Go.
\end{cases}
\]
Then $g_0 \in C_c(G,\F_d) = A_{\F_d}(G,\sigma)$ by \cref{lemma: functions in twisted Steinberg algebras}\cref{item: functions are C_c and LC}, and $\supp(g_0) = \Go \cap \supp(g)$. Define $H \coloneqq \supp(g - g_0) \subseteq G {\setminus} \Go$. \Cref{eqn: g(s(alpha)) is nonzero} implies that $s(\alpha) \in \supp(g_0)$. Since $G$ is ample and effective, \cite[Lemma~3.1]{BCFS2014} implies that there is a nonempty compact open subset $V$ of $\supp(g_0) \cap s(D_0)$ such that $VHV = \varnothing$. Therefore, since $\supp(1_V (g - g_0) 1_V) \subseteq VHV$, we have $1_V (g - g_0) 1_V = 0$, and hence \cref{eqn: g(s(alpha)) is nonzero} implies that
\begin{equation} \label{eqn: 1_V g 1_V}
1_V \, g \, 1_V = 1_V \, g_0 \, 1_V = c_{D_0} \, 1_V.
\end{equation}
Thus, using that $\pi(g) = 0$, we deduce from \cref{eqn: 1_V g 1_V} that
\[
\pi(1_V) = c_{D_0}^{-1} \, \pi(c_{D_0} \, 1_V) = c_{D_0}^{-1} \, \pi(1_V) \, \pi(g) \, \pi(1_V) = 0,
\]
as required.
\end{proof}

Given a groupoid $G$, we call a subset $U \subseteq \Go$ \hl{invariant} if, for any $\gamma \in G$, we have
\[
s(\gamma) \in U \iff r(\gamma) \in U.
\]
We say that a topological groupoid $G$ is \hl{minimal} if $\Go$ has no nontrivial open invariant subsets. Equivalently, $G$ is minimal if and only if $\overline{s(r^{-1}(x))} = \Go$ for every $x \in \Go$.

\begin{thm} \label{thm: simplicity}
Let $\F_d$ be a discrete field, let $G$ be an effective ample Hausdorff groupoid, and let $\sigma\colon \Gc \to \F_d^\times$ be a continuous $2$-cocycle. Then $A_{\F_d}(G,\sigma)$ is simple if and only if $G$ is minimal.
\end{thm}

\begin{proof}
Suppose that $G$ is minimal, and let $I$ be a nonzero ideal of $A_{\F_d}(G,\sigma)$. Then $I$ is the kernel of some noninjective ring homomorphism of $A_{\F_d}(G,\sigma)$, and so \cref{thm: CK uniqueness} implies that there is a compact open subset $V \subseteq \Go$ such that $1_V \in I$. We claim that the ideal generated by $1_V$ is the whole of $A_{\F_d}(G,\sigma)$. Since the twisted convolution product of characteristic functions on the unit space is the same as the untwisted convolution product, the proof follows directly from the arguments used in the proof of \cite[Theorem~4.1]{CE-M2015}.

For the converse, suppose that $G$ is not minimal. Then there exists a nonempty open invariant subset $U \subsetneq \Go$. The set
\[
G_U \coloneqq s^{-1}(U) = \{ \gamma \in G : s(\gamma) \in U\} = \{\gamma \in G : r(\gamma) \in U \}
\]
is a proper open subgroupoid of $G$, and so we can view $I \coloneqq A_{\F_d}\big(G_U,\sigma\restr{G_U^{(2)}}\big)$ as a proper subset of $A_{\F_d}(G,\sigma)$. Since $U$ is a nonempty open set and $G$ is ample, we can find a nonempty compact open bisection $B$ of $G$ contained in $U$, and thus $I \ne \{0\}$, because $1_B \in I$. We claim that $I$ is an ideal of $A_{\F_d}(G,\sigma)$. Since the vector-space operations are defined pointwise, it is straightforward to check that $I$ is a subspace. To see that $I$ is an ideal, fix $f \in I$ and $g \in A_{\F_d}(G,\sigma)$. Since $U$ is invariant, we have
\[
\supp(fg) \subseteq \supp(f) \, \supp(g) \subseteq G_U \, G \subseteq G_U,
\]
and so $fg \in I$. Similarly, $gf \in I$, and thus $I$ is an ideal. (In fact, if $A_{\F_d}(G,\sigma)$ is a $*$-algebra, then $I$ is a $*$-ideal.)
\end{proof}

\begin{remark}
By \cite[Theorem~4.1]{BCFS2014}, the untwisted complex Steinberg algebra $A(G)$ is simple if and only if $G$ is minimal and effective. Note that \cref{thm: simplicity} does not give necessary and sufficient conditions on $G$ and $\sigma$ for simplicity of twisted Steinberg algebras. This is a hard problem. We expect, as in the C*-setting of \cite[Remark~8.3]{KPS2015}, that there exist simple twisted Steinberg algebras for which the groupoid $G$ is not effective.
\end{remark}

\section{Gradings and a graded uniqueness theorem} \label{section: gradings}

In this section we describe the graded structure that twisted Steinberg algebras inherit from the underlying groupoid, and we prove a graded uniqueness theorem. The arguments are similar to those used in the untwisted setting (see \cite{CE-M2015}). Let $\Gamma$ be a discrete group, and suppose that $c\colon G \to \Gamma$ is a continuous groupoid homomorphism (or \hl{$1$-cocycle}). Then we call $G$ a \hl{$\Gamma$-graded groupoid}, and we define $G_\gamma \coloneqq c^{-1}(\gamma)$ for each $\gamma \in \Gamma$. Since $c$ is continuous and $\Gamma$ is discrete, each $G_\gamma$ is clopen. Since $c$ is a homomorphism, we have
\[
G_\gamma^{-1} = G_{\gamma^{-1}} \quad \text{ and } \quad G_\zeta \, G_\eta \subseteq G_{\zeta\eta}
\]
for all $\gamma, \zeta, \eta \in \Gamma$. Note that all groupoids are graded with respect to the groupoid homomorphism into the trivial group.

\begin{prop} \label{prop: graded}
Let $G$ be an ample Hausdorff groupoid, and let $\sigma\colon \Gc \to R^\times$ be a continuous $2$-cocycle. Suppose that $\Gamma$ is a discrete group and $c\colon G \to \Gamma$ is a continuous groupoid homomorphism. For each $\gamma \in \Gamma$, define the set of \hl{homogeneous elements of degree $\gamma$} by
\[
A_R(G,\sigma)_{\gamma} \coloneqq \{ f \in A_R(G,\sigma) \,:\, \supp(f) \subseteq G_{\gamma} \}.
\]
Then $A_R(G,\sigma)$ is a $\Gamma$-graded algebra.
\end{prop}

\begin{proof} It is clear that $A_R(G,\sigma)_\gamma$ is an $R$-submodule of $A_R(G,\sigma)$, for each $\gamma \in \Gamma$. Since $A_R(G,\sigma)$ and $A_R(G)$ agree as $R$-modules, the argument used in the proof of \cite[Lemma~2.2]{CE-M2015} can be used to show that every $f \in A_R(G,\sigma)$ can be expressed as an $R$-linear combination of homogeneous elements. Thus, to see that
\[
A_R(G,\sigma) = \bigoplus_{\gamma \in \Gamma} A_R(G,\sigma)_{\gamma},
\]
it suffices to show that any finite collection
\[
\{ f_i \in A_R(G,\sigma)_{\gamma_i} \,:\, 1 \le i \le n, \text{ and each $\gamma_i$ is distinct from the others} \}
\]
is linearly independent. But this is clear, because $\supp(f_i) \cap \supp(f_j) = \varnothing$ when $i \ne j$. Fix $\zeta, \eta \in \Gamma$. For all $f \in A_R(G,\sigma)_\zeta$ and $g \in A_R(G,\sigma)_\eta$, we have
\[
\supp(fg) \subseteq \supp(f) \, \supp(g) \subseteq G_\zeta \, G_\eta \subseteq G_{\zeta\eta},
\]
and hence
\[
A_R(G,\sigma)_\zeta \, A_R(G,\sigma)_\eta \subseteq A_R(G,\sigma)_{\zeta\eta}. \qedhere
\]
\end{proof}

As in the untwisted setting \cite[Theorem~3.4]{CE-M2015}, the graded uniqueness theorem follows from the Cuntz--Krieger uniqueness theorem. Note that if $e$ is the identity of $\Gamma$, then $G_e$ is a clopen subgroupoid of $G$, and so we can identify $A_R(G,\sigma)_e$ with $A_R(G_e,\sigma)$, just as we can identify $A_R(G_e)$ with $A_R(G)_e$.

\begin{thm}[Graded uniqueness theorem] \label{thm: graded uniqueness}
Let $\F_d$ be a discrete field, let $G$ be an ample Hausdorff groupoid, and let $\sigma\colon \Gc \to \F_d^\times$ be a continuous $2$-cocycle. Let $\Gamma$ be a discrete group with identity $e$, and suppose that $c\colon G \to \Gamma$ is a continuous groupoid homomorphism such that the subgroupoid $G_e$ is effective. Suppose that $Q$ is a $\Gamma$-graded ring and that $\pi\colon A_{\F_d}(G,\sigma) \to Q$ is a graded ring homomorphism. Then $\pi$ is injective if and only if $\pi(1_K) \ne 0$ for every nonempty compact open subset $K$ of $\Go$.
\end{thm}

\begin{proof}
It is clear that if $\pi$ is injective, then $\pi(1_K) \ne 0$ for every nonempty compact open subset $K$ of $\Go$. Suppose that $\pi$ is not injective. We claim that there exists $f \in A_{\F_d}(G_e,\sigma)$ such that $f \ne 0$ and $\pi(f) = 0$. To see this, fix $g \in \ker(\pi)$ such that $g \ne 0$. By the proof of \cref{prop: graded}, $g$ can be expressed as a finite sum of homogeneous elements; that is, $g = \sum_{\gamma \in F} g_{\gamma}$, where $F$ is a finite subset of $\Gamma$, and $g_{\gamma} \in A_{\F_d}(G,\sigma)_{\gamma}$ for each $\gamma \in F$. Thus,
\[
\sum_{\gamma \in F} \pi(g_{\gamma}) = \pi\Big(\sum_{\gamma \in F} g_{\gamma}\Big) = \pi(g) = 0.
\]
Since $\pi$ is graded, we have $\pi(g_{\gamma}) \in Q_\gamma$ for each $\gamma \in \Gamma$. Thus $\pi(g_{\gamma}) = 0$ for each $\gamma \in \Gamma$, because elements of different graded subspaces of $Q$ are linearly independent. Since $g \ne 0$, we can choose $\gamma \in F$ such that $g_{\gamma} \ne 0$. Since $g_{\gamma}$ is locally constant and $G_\gamma$ is open, there exists a compact open bisection $B \subseteq G_{\gamma}$ such that $g_{\gamma}(B) = \{k\}$, for some $k \in \F_d {\setminus} \{0\}$. Define $f \coloneqq 1_{B^{-1}} \, g_{\gamma}$. Since $\pi$ is a homomorphism and $G$ is graded, we have $f \in A_{\F_d}(G_e,\sigma) \cap \ker(\pi)$. For all $\alpha \in B$, we have
\[
f(s(\alpha)) = f(\alpha^{-1}\alpha) = \sigma(\alpha^{-1},\alpha) \, 1_{B^{-1}}(\alpha^{-1}) \, g_{\gamma}(\alpha) = \sigma(\alpha,\alpha^{-1}) \, k \ne 0,
\]
and hence $f \ne 0$. Thus the restriction $\pi_e$ of $\pi$ to $A_{\F_d}(G_e,\sigma)$ is not injective.

Since $\Go \subseteq G_e$ and we have assumed that the groupoid $G_e$ is effective, we can apply \cref{thm: CK uniqueness} to the restricted homomorphism $\pi_e$ to obtain a nonempty compact open subset $K \subseteq \Go$ such that $\pi(1_K) = 0$, as required.
\end{proof}

\vspace{4ex}

\begin{acknowledgements}
This research collaboration began as part of the project-oriented workshop ``Women in Operator Algebras'' (18w5168) in November 2018, which was funded and hosted by the Banff International Research Station. The attendance of the first-named author at this workshop was supported by an AustMS WIMSIG Cheryl E.\ Praeger Travel Award, and the attendance of the third-named author was supported by SFB 878 Groups, Geometry \& Actions. The research was also funded by the Australian Research Council grant DP170101821, and by the Deutsche Forschungsgemeinschaft (DFG, German Research Foundation) under Germany's Excellence Strategy -- EXC 2044 -- 390685587, Mathematics M\"{u}nster -- Dynamics -- Geometry -- Structure, and under SFB 878 Groups, Geometry \& Actions. The authors would like to thank Gilles de Castro, Aidan Sims, and Ben Steinberg for helpful feedback on the first version of this paper.
\end{acknowledgements}

\vspace{2ex}

\bibliographystyle{amsplain}
\makeatletter\renewcommand\@biblabel[1]{[#1]}\makeatother
\bibliography{ACCLMR_references}

\providecommand{\bysame}{\leavevmode\hbox to3em{\hrulefill}\thinspace}
\providecommand{\MR}{\relax\ifhmode\unskip\space\fi MR }
\providecommand{\MRhref}[2]{%
  \href{http://www.ams.org/mathscinet-getitem?mr=#1}{#2}
}
\providecommand{\href}[2]{#2}
\begin{thebibliography}{10}

\bibitem{APCaHR2013}
G.~{\noopsort{Aranda}{Aranda Pino}}, J.~Clark, A.~{\noopsort{Huef}{an Huef}},
  and I.~Raeburn, \emph{Kumjian--pask algebras of higher-rank graphs}, Trans.
  Amer. Math. Soc. \textbf{365} (2013), 3613--3641.

\bibitem{AB2018}
B.~Armstrong and N.~Brownlowe, \emph{Product-system models for twisted
  {C*-algebras} of topological higher-rank graphs}, J. Math. Anal. Appl.
  \textbf{466} (2018), 1443--1475.

\bibitem{BL2017}
S.~Barlak and X.~Li, \emph{Cartan subalgebras and the {UCT} problem}, Adv.
  Math. \textbf{316} (2017), 748--769.

\bibitem{Boenicke2021}
C.~B\"onicke, \emph{{K-theory} and homotopies of twists on ample groupoids}, J.
  Noncommut. Geom. \textbf{15} (2021), 195--222.

\bibitem{BCFS2014}
J.H. Brown, L.O. Clark, C.~Farthing, and A.~Sims, \emph{Simplicity of algebras
  associated to \'etale groupoids}, Semigroup Forum \textbf{88} (2014),
  433--452.

\bibitem{BaH2014}
J.H. Brown and A.~{\noopsort{Huef}{an Huef}}, \emph{Decomposing the
  {C*-algebras} of groupoid extensions}, Proc. Amer. Math. Soc. \textbf{142}
  (2014), 1261--1274.

\bibitem{Brown1982}
K.S. Brown, \emph{{Cohomology of Groups}}, Graduate Texts in Mathematics,
  vol.~87, Springer, New York, 1982.

\bibitem{CE-M2015}
L.O. Clark and C.~Edie-Michell, \emph{Uniqueness theorems for {S}teinberg
  algebras}, Algebr. Represent. Theory \textbf{18} (2015), 907--916.

\bibitem{CEP2018}
L.O. Clark, R.~Exel, and E.~Pardo, \emph{A generalized uniqueness theorem and
  the graded ideal structure of {Steinberg} algebras}, Forum Math. \textbf{30}
  (2018), 533--552.

\bibitem{CFST2014}
L.O. Clark, C.~Farthing, A.~Sims, and M.~Tomforde, \emph{A groupoid
  generalization of {Leavitt} path algebras}, Semigroup Forum \textbf{89}
  (2014), 501--517.

\bibitem{CaHR2013}
L.O. Clark, A.~{\noopsort{Huef}{an Huef}}, and I.~Raeburn, \emph{The
  equivalence relations of local homeomorphisms and {F}ell algebras}, New York
  J. Math. \textbf{19} (2013), 367--394.

\bibitem{CP2017}
L.O. Clark and Y.E.P. Pangalela, \emph{{Kumjian--Pask} algebras of finitely
  aligned higher-rank graphs}, J. Algebra \textbf{482} (2017), 364--397.

\bibitem{ER2018}
C.~Eckhardt and S.~Raum, \emph{{C*-superrigidity} of {$2$-step} nilpotent
  groups}, Adv. Math. \textbf{338} (2018), 175--195.

\bibitem{Exel2008}
R.~Exel, \emph{Inverse semigroups and combinatorial {C*}-algebras}, Bull. Braz.
  Math. Soc. (N.S.) \textbf{39} (2008), 191--313.

\bibitem{Exel2010}
\bysame, \emph{Reconstructing a totally disconnected groupoid from its ample
  semigroup}, Proc. Amer. Math. Soc. \textbf{138} (2010), 2991--3001.

\bibitem{FMY2005}
C.~Farthing, P.S. Muhly, and T.~Yeend, \emph{Higher-rank graph {C*-algebras}:
  an inverse semigroup and groupoid approach}, Semigroup Forum \textbf{71}
  (2005), 159--187.

\bibitem{Gillaspy2015}
E.~Gillaspy, \emph{{K-theory} and homotopies of {$2$-cocycles} on higher-rank
  graphs}, Pacific J. Math. \textbf{278} (2015), 407--426.

\bibitem{Kumjian1986}
A.~Kumjian, \emph{On {C*}-diagonals}, Canad. J. Math. \textbf{38} (1986),
  969--1008.

\bibitem{KP2000}
A.~Kumjian and D.~Pask, \emph{Higher rank graph {C*}-algebras}, New York J.
  Math. \textbf{6} (2000), 1--20.

\bibitem{KPRR1997}
A.~Kumjian, D.~Pask, I.~Raeburn, and J.~Renault, \emph{Graphs, groupoids, and
  {Cuntz--Krieger} algebras}, J. Funct. Anal. \textbf{144} (1997), 505--541.

\bibitem{KPS2012}
A.~Kumjian, D.~Pask, and A.~Sims, \emph{Homology for higher-rank graphs and
  twisted {C*}-algebras}, J. Funct. Anal. \textbf{263} (2012), 1539--1574.

\bibitem{KPS2013}
\bysame, \emph{On the {K-theory} of twisted higher-rank-graph {C*}-algebras},
  J. Math. Anal. Appl. \textbf{401} (2013), 104--113.

\bibitem{KPS2015}
\bysame, \emph{On twisted higher-rank graph {C*}-algebras}, Trans. Amer. Math.
  Soc. \textbf{367} (2015), 5177--5216.

\bibitem{KPS2016}
\bysame, \emph{Simplicity of twisted {C*}-algebras of higher-rank graphs and
  crossed products by quasifree actions}, J. Noncommut. Geom. \textbf{10}
  (2016), 515--549.

\bibitem{Li2019}
X.~Li, \emph{Every classifiable simple {C*-algebra} has a {Cartan} subalgebra},
  Invent. Math. (2019), 1--46.

\bibitem{MS2018}
L.~Margolis and O.~Schnabel, \emph{Twisted group ring isomorphism problem}, Q.
  J. Math. \textbf{69} (2018), 1195--1219.

\bibitem{PR1992}
J.A. Packer and I.~Raeburn, \emph{On the structure of twisted group
  {C*-algebras}}, Trans. Amer. Math. Soc. \textbf{334} (1992), 685--718.

\bibitem{Paterson1999}
A.L.T. Paterson, \emph{{Groupoids, Inverse Semigroups, and their Operator
  Algebras}}, Progress in Mathematics, vol. 170, Birkh\"auser Boston, Inc.,
  Boston, MA, 1999.

\bibitem{RSY2004}
I.~Raeburn, A.~Sims, and T.~Yeend, \emph{The {C*}-algebras of finitely aligned
  higher-rank graphs}, J. Funct. Anal. \textbf{213} (2004), 206--240.

\bibitem{Renault1980}
J.~Renault, \emph{A groupoid approach to {C*-algebras}}, Lecture Notes in
  Math., vol. 793, Springer-Verlag, New York, 1980.

\bibitem{Renault2008}
\bysame, \emph{Cartan subalgebras in {C*-algebras}}, Irish Math. Soc. Bulletin
  \textbf{61} (2008), 29--63.

\bibitem{Sims2020}
A.~Sims, \emph{Hausdorff \'etale groupoids and their {C*-algebras}}, Operator
  algebras and dynamics: groupoids, crossed products, and Rokhlin dimension
  (F.~Perera, ed.), Advanced Courses in Mathematics, CRM Barcelona,
  Birkh\"auser, 2020.

\bibitem{SWW2014}
A.~Sims, B.~Whitehead, and M.F. Whittaker, \emph{Twisted {C*-algebras}
  associated to finitely aligned higher-rank graphs}, Documenta Math.
  \textbf{19} (2014), 831--866.

\bibitem{Steinberg2010}
B.~Steinberg, \emph{A groupoid approach to discrete inverse semigroup
  algebras}, Adv. Math. \textbf{223} (2010), 689--727.

\bibitem{Tu1999}
J.L. Tu, \emph{La conjecture de {B}aum--{C}onnes pour les feuilletages
  moyennables}, K-Theory \textbf{17} (1999), 215--264.

\bibitem{Yeend2007}
T.~Yeend, \emph{Groupoid models for the {C*}-algebras of topological
  higher-rank graphs}, J. Operator Theory \textbf{57} (2007), 95--120.

\end{thebibliography}

\newpage
\end{document}